\documentclass[letterpaper, dvipsnames, 12pt]{article}
\usepackage[margin=1in]{geometry}
\usepackage[utf8]{inputenc}
\usepackage[T1]{fontenc}
\usepackage{amsmath,amsthm,amssymb}

\usepackage{graphicx}
\usepackage{tikz}
\usepackage{hyperref}
\usepackage{color}
\usepackage{enumerate}
\usepackage{multicol}
\usepackage{xcolor}
\usepackage{amsrefs}
\newcommand{\SL}{\sum\limits_}
\newcommand{\PL}{\prod\limits_}

\newcommand{\R}{\mathbb{R}}
\newcommand{\N}{\mathbb{N}}
\newcommand{\Z}{\mathbb{Z}}
\newcommand{\PR}{\mathbb{P}}
\newcommand{\E}{\mathbb{E}}

\usepackage[autostyle,english=american]{csquotes}
\MakeOuterQuote{"}

\newtheorem{ccounter}{ccounter}[section]
\newtheorem{thm}[ccounter]{Theorem}
\newtheorem{lem}[ccounter]{Lemma}
\newtheorem{cor}[ccounter]{Corollary}
\newtheorem{defn}[ccounter]{Definition}
\newtheorem{prop}[ccounter]{Proposition}

\newtheorem{rem}[ccounter]{Remark}
\newtheorem{claim}[ccounter]{Claim}
\newtheorem{obs}[ccounter]{Observation}

\title{Polynomial mixing of the critical Ising model on sparse Erdos-Renyi graphs}
\author{Kyprianos-Iason Prodromidis\thanks{Department of Mathematics, Princeton University, Email: kp2702@princeton.edu} \and Allan Sly\thanks{Department of Mathematics, Princeton University, Email: allansly@princeton.edu}}
\date{}
\begin{document}
\maketitle
\begin{abstract}
We consider the stochastic Ising model on sparse Erdös-Rényi graphs $G(n,d/n)$ with $d>1$ at the critical temperature $\beta_c=\tanh^{-1}(d^{-1})$ and prove that with high probability, the mixing time is at most polynomial in $n$. Our approach combines the recent stochastic localization framework of Chen and Eldan, which yields spectral gap bounds in the well-behaved bulk of the graph, together with classical results on the relaxation time of Glauber dynamics on trees to handle regions where we cannot apply the Chen–Eldan method  directly because of atypically large local neighborhoods. 
\end{abstract}
\vspace{1em}
\section{Introduction}
A key question in the study of the Ising model concerns how long it takes to reach equilibrium when spins evolve according to the Glauber dynamics. This Markov chain, introduced by Glauber \cite{Glauber} in 1963, provides a natural Markov process in which spins evolve according to a local update rule. The resulting chain is reversible with respect to the Gibbs distribution, and its mixing time depends both on the geometry of the underlying graph and the temperature regime. Understanding the speed of convergence has important implications for sampling and algorithmic approximations.

For each graph family, it is expected that there exists a critical value $\beta_c$ that separates three qualitatively different regimes: In high temperatures, when $\beta < \beta_c$, fast convergence occurs due to spatial mixing properties that hold in this regime. In low temperatures, when $\beta > \beta_c$, convergence is slow, typically exponential in some parameter that depends on the geometry of the graph. At the critical temperature, when $\beta = \beta_c$, mixing is typically polynomial but suboptimal, due to the lack of exponential decay of correlations. This behavior has been verified in different families of graphs which we will briefly summarize.\\

\textbf{Complete Graph}: First we examine the case $G=K_n$. For $\beta<\beta_c$, as proven in \cite{sub-complete}, the mixing time is $\Theta(n\log(n))$. For $\beta>\beta_c$, mixing is exponentially slow in the volume of the graph \cite{super-complete}. In \cite{critical-complete} it was shown that at $\beta=\beta_c$, the mixing time is $\Theta(n^{3/2})$.\\

\textbf{Trees}: The case of trees was extensively studied in \cite{BCMP} where it was shown that $\beta_c$ corresponds to the reconstruction threshold. In this case the crucial parameter is the height $h$ of the tree. The relaxation time in the high-temperature regime is bounded, while in the low-temperature regime it is exponential in $h$. The critical temperature was settled later, in \cite{DLP}, where polynomial relaxation was shown.\\

\textbf{Lattices}: The case of lattices, $G=\Z^d$ for $d\ge2$, poses additional challenges due to the more complicated geometry. The mixing rates in the high and low temperature regimes have been established since the early 2000s, see \cite{sub-lattices} for the high-temperature regime and \cite{2d-super}, \cite{Pisztora}, \cite{Bodineau} for the low-temperature regime where the mixing time is exponential in the surface area of a cube. On the other hand, proving polynomial mixing at the critical temperature has proven to be a much harder problem. For $\Z^2$, polynomial mixing was proven by Lubetzky and the second author \cite{Lubetzky-Sly}, and the proof relies heavily on the planarity of dimension 2. Recently, breakthrough works by Bauerschmidt and Dagallier \cite{BD}, as well as Chen and Eldan \cite{StochLoc}, established new tools to bound the log-Sobolev constant of Glauber dynamics for the Ising models. In both of these works, the inequalities shown involve the \textit{susceptibility} of the model, defined as
\begin{equation}\label{suscept}
\chi_\beta=\sup\limits_{x\in G}\SL{y\in G}\E_\beta(\sigma_x\cdot\sigma_y).
\end{equation}
Showing that
$\chi_\beta \leq C(\beta_c-\beta)^{-1}$ implies a polynomial bound on the log-Sobolev constant of the Glauber Dynamics. Known bounds on the decay of correlations of the Ising model imply polynomial mixing when $d\ge5$. These do no hold for $d\leq 4$ and so the question of polynomial mixing for the cases $d=3,4$ remains an open question.\\

\textbf{Sparse Random Graphs}: Another case that has attracted interest more recently are sparse random graphs. In the work of Mossel and the second author \cite{Mossel&Sly}, it was shown that the mixing time in the high-temperature regime is optimal, i.e. of order $\Theta(n\log(n))$ for random regular graphs and of order $n^{1+\Theta(1/\log\log(n))}$ for Erdős–Rényi graphs. Around the same time, exponential mixing in the size of the graph was established both in the case of random regular graphs \cite{G-M} and in the case of sparse Erdös-Rényi graphs \cite{D-M}. As noted in \cite{Polchinski-eq} (see also \cite{CCYZ}, \cite{PS} for the full details), the techniques developed by Bauerschmidt, Dagallier, Chen and Eldan can be used to prove polynomial mixing in the case of regular graphs as well.\\\\
In this paper, we analyze the Glauber dynamics at criticality for sparse Erdős–Rényi graphs and prove polynomial mixing with high probability.
\begin{thm}\label{main-theorem}
Consider the mixing time $t_{\text{mix}}$ of the Glauber dynamics for the Ising model on $G(n,d/n)$, with $d\tanh(\beta)=1$. There exists a constant $D=D(d)$, independent of $n$, such that
$$\lim\limits_{n\to\infty}\PR(t_{\text{mix}}>n^D)=0.$$
\end{thm}
The main challenge compared to the regular graph case is the presence of high-degree vertices, which preclude the direct application of the susceptibility bounds discussed earlier. However, for most vertices $x$, the quantity $\SL{y\in G}\E_\beta(\sigma_x\cdot\sigma_y)$ that appears in (\ref{suscept}) is bounded above by $C(\beta_c-\beta)^{-1}$. Essentially, we show that as long as the vertices where this decay fails are sufficiently rare and well-separated, the framework developed in \cite{StochLoc} remains applicable and implies polynomial mixing. In fact, Theorem~\ref{main-theorem} will follow as a consequence of the  following theorem, once we prove that the properties listed hold with high probability for a $G(n,d/n)$.
\begin{thm}\label{main-thm-determ}
Let $G$ be a graph on $n$ vertices. Assume that for some $\delta,C,C',C''>0$ (possibly depending on $d$ but independent of $n$) there is a partition of the vertices of $G$ into two sets $A,B$ such that:
\begin{enumerate}
\item 
We denote $H$ as the graph induced by the edges which have at least one endpoint in $A$.  The connected components of $H$ have tree excess at most 1.
\item 
The inequality
$$\max\limits_{\gamma\subset H}\SL{v\in \gamma}\deg(v)\le C\log(n),$$
holds, where the maximum runs over all Self-Avoiding paths in $H$.
\item 
In every Self-Avoiding path $p$ in $G$ of length at least $\ell\ge\delta\log_d(n)$, at least $0.6\ell$ of its edges connect vertices in $B$.
\item 
For each $1\le\ell\le n$ and $v\in G$, let $S_v^{(\ell)}$ be the number of Self-Avoiding Walks of length $\ell$  in $G$ originating from $v$. Then, $S_v^{(\ell)}\le C'\cdot d^\ell$, for every $v\in B,\ell\le\delta\log_d(n)$ and
$$\SL{\ell=1}^nd^{-\ell}\cdot\dfrac{1}{\ell}\cdot\sup\limits_{v\in B} S_v^{(\ell)}\le C''\log(n).$$
\end{enumerate}
Then, for some $D=D(\delta,C,C',C'')>0$, the mixing time $t_{\text{mix}}$ of the Glauber Dynamics for the Ising Model with inverse temperature $\beta=\tanh^{-1}(d^{-1})$ satisfies
$$t_{\text{mix}}\le n^D.$$
\end{thm}
\begin{prop}\label{ER-whp-prop}
For each $d>1$ there exist $\delta,C,C',C''>0$ such that a graph $G$ drawn from $G(n,d/n)$ admits, with high probability, a partition of the vertices in two parts $A$ and $B$ that satisfies all the properties listed in Theorem \ref{main-thm-determ}.
\end{prop}
\subsection{Outline of the paper}
The paper is structured as follows. In Section 2, we introduce the necessary definitions and the notation with which we will be working with. In Section 3, we prove Theorem~\ref{main-thm-determ}. Our method seeks to use the tools developed in \cite{StochLoc} for the Glauber dynamics on $B$, together with a theorem of \cite{Mossel-Sly}, which treats the Glauber dynamics on $H$. More specifically: 
\begin{itemize}
\item 
Using Proposition~4 from~\cite{Mossel-Sly}, we show that the relaxation time of the Glauber dynamics on $H$ is at most polynomial in $n$. It will be crucial that $H$ has logarithmic small diameter and that the tree excess of each connected component is at most 1.
\item 
We compare the Glauber dynamics on $G$ to the Restricted dynamics on $B$, and show that their spectral gaps are at most a polynomial factor away from each other. The argument to achieve that is roughly as follows: We run the Glauber Dynamics on $G$, so that the rate at which we update vertices in $A$ is much higher than the rate at which we update vertices in $B$. This way, between every update in $B$, the configuration on $A$ is really well mixed, and at the same time the Dirichlet form (and therefore the spectral gap) only loses a polynomial factor.
\item 
Using the tools developed in \cite{StochLoc}, we show that there is only a polynomial factor separating the spectral gaps of the Restricted dynamics on $B$ and the Glauber dynamics on $H$.\\
\end{itemize}
Section 4 is devoted to the proof of Proposition \ref{ER-whp-prop}. The way we decompose the graph into two parts is based on whether the initial branching of the vertices is faster than usual or not.  We let $$B(x,r)=\{y\in G:\text{d}_G(x,y)\le r\}$$ denote the ball of radius $r$ around $x$ and $\partial B(x,r)=B(x,r)\setminus B(x,r-1)$ its boundary.
\begin{defn}\label{B-defn}
For a vertex $x$,
$$x\in B_{\delta,C}\Leftrightarrow \max\limits_{r\le\delta\log_d(n)}|\partial B(x,r)|\cdot d^{-r}\le C.$$
\end{defn}
For the sake of concreteness, we can set $\delta=10^{-1}$, and $C>0$ will be a sufficiently large constant chosen later. From now on, we omit the dependence of $B$ on $\delta$ and $C$, and let $A=G\setminus B$.
\begin{rem}\label{balls-slowly}
Assume $x\in B$. Then, for each $r\le\delta\log_d(n)$,
$$|B(x,r)|=\SL{r'=0}^r|\partial B(x,r')|\le C\SL{r'=0}^rd^{r'}\le C\cdot(1-d^{-1})^{-1}\cdot d^r.$$
Therefore, as expected, the balls centered at $x$ also do not grow very quickly.
\end{rem}
In Subsection \ref{4.1}, we prove that Properties 1, 2 and 3 from Theorem \ref{main-thm-determ} hold with high probability. Subsection~\ref{4.2} is devoted to proving an important lemma on Non-Backtracking walks, which allows us to avoid a union over vertices while trying to prove Property~4. Next, Subsection \ref{4.3} contains several lemmas on the neighborhood structure in $G(n,d/n)$, necessary for our calculations. The proof of Property~4 is concluded in Subsection \ref{4.4}, where there are moment calculations regarding numbers of walks in $G(n,d/n)$.
\section{Preliminaries}
In this section, we provide some important standard definitions and well-known results which will be relevant throughout this paper.
\subsection{Continuous-time Markov Chains}
In this paper, we are working with continuous-time Markov chains. For completeness, we provide the relevant definitions.
\begin{defn}\label{CTMC}
\begin{itemize}
\item
Let $\mathcal{X}$ be a finite state space and $Q$ be a matrix indexed by $\mathcal{X}\times\mathcal{X}$, such that $q(x,y)\ge0$ for $x\neq y$ and
$$q(x,x)=-\SL{y\neq x}q(x,y),\ \ \text{for all}\ \ x\in\mathcal{X}.$$
The Markov chain $(X_t)_{t\ge0}$ with jump matrix $Q$ is the $\mathcal{X}$-valued process which evolves according to the following rule: When currently in state $x$, each $y\neq x\in\mathcal{X}$ is assigned a Poisson clock at rate $q(x,y)\ge0$, with the clocks being independent of each other. As soon as the first clock rings, the process moves to the respective state. The total jump rate from $x\in\mathcal{X}$ is
$$q(x):=\SL{y\neq x}q(x,y).$$
\item
The heat kernel $H_t$ is defined by
$$H_t(x,y):=\PR(X_t=y\ |\ X_0=x).$$
As in discrete-time Markov Chains, there exists a stationary distribution $\pi$ such that 
$$\sup\limits_{x\in\mathcal{X}}\text{d}_{\text{TV}}(H_t(x,\cdot),\pi)\xrightarrow[t\to\infty]{}0,$$
where for two probability measures on $\mathcal{X}$, $\text{d}_{\text{TV}}(\mu,\nu)$ is the \textbf{Total Variation Distance} between $\mu,\nu$, i.e.
$$\text{d}_{\text{TV}}(\mu,\nu)=\sup\limits_{A}|\mu(A)-\nu(A)|=\dfrac{1}{2}\SL{x\in\mathcal{X}}|\mu(x)-\nu(x)|.$$
\item 
The mixing time of this Markov Chain is defined to be
$$t_{\text{mix}}(\varepsilon)=\inf\left\{t\ge0:\ \sup\limits_{x\in\mathcal{X}}\text{d}_{\text{TV}}(H_t(x,\cdot),\pi)\le\varepsilon\right\}.$$
We also set $t_{\text{mix}}:=t_{\text{mix}}(1/4)$.
\item 
Assume $(X_t)_{t\ge0}$ is reversible and let $0=\lambda_1<\lambda_2\le\cdots\le\lambda_{|\mathcal{X}|}$ be the eigenvalues of the matrix $-Q$. The spectral gap of $(X_t)_{t\ge0}$, denoted $\text{gap}(X_t)$ is defined to be $\lambda_2$. An equivalent definition is
$$\text{gap}(X_t):=\inf\limits_{f\neq c}\dfrac{\mathcal{E}(f,f)}{\text{Var}_\pi(f)},$$
where $\mathcal{E}(f,f)$ is the Dirichlet form associated to $(X_t)_{t\ge0}$, defined as
$$\mathcal{E}(f,f):=-\langle f,Qf\rangle_\pi=\dfrac{1}{2}\SL{x,y\in\mathcal{X}}\pi(x)q(x,y)(f(x)-f(y))^2,$$
for any function $f:\mathcal{X}\to\R$. The inverse of the spectral gap is usually referred to as the \textbf{relaxation time} of the Markov chain.
\end{itemize}
\end{defn}
We also mention an important and well-known theorem, which helps relate the spectral gap and the mixing time of a Markov chain. For a proof, we refer to \cite{AF}, Lemma 4.23.
\begin{thm}\label{rel-spec-mix}
For some $c_1,c_2>0$ and a reversible Markov Chain with spectral gap $\gamma$,
$$c_1\cdot\gamma^{-1}\cdot\log\left(\dfrac{1}{2\varepsilon}\right)\le t_{\text{mix}}(\varepsilon)\le c_2\cdot\gamma^{-1}\cdot\log\left(\dfrac{1}{2\varepsilon\pi_{\min}}\right).$$
\end{thm}
\subsection{Ising Model and Glauber Dynamics}
\begin{defn}
\begin{itemize}
\item 
Let $G$ be a graph with vertex set $V$. The Ising model on this graph with inverse temperature $\beta$ and external field $h:V\to[-\infty,\infty]$ is the probability measure $\mu$ on $\{-1,1\}^V$ satisfying
$$\mu_{\beta,h}(\sigma)=\dfrac{1}{Z_{\beta,G,h}}\exp\left(\beta\SL{u\sim v}\sigma_u\sigma_v+\SL{v\in V}h(v)\sigma_v\right).$$
By $\mu_\beta$ we will just denote the Ising model with $h=0$.
\item
Fix a measure $\nu$ on $\{-1,1\}^V$. For a configuration $\sigma\in\{-1,1\}^V$ and for some $v\in V$, let $\sigma^{\oplus v}$ be the configuration $\sigma$ in which the spin at vertex $v$ is flipped. For two configurations $\sigma,\tau\in\{-1,1\}^V$, we write $\sigma\sim\tau$ if $\tau=\sigma^{\oplus v}$ for some $v\in V$. The continuous time Glauber Dynamics on $G$ for $\nu$ at rate 1 is a Markov chain with transition rates, for $\sigma\neq\tau$, equal to
$$q(\sigma,\tau)=\begin{cases}
\dfrac{\nu(\tau)}{\nu(\sigma)+\nu(\tau)},\ \ \text{if}\ \ \sigma \sim \tau
\\ \\
0,\ \ \text{otherwise}.
\end{cases}$$
\end{itemize}
\end{defn}
Our analysis in Chapter 3, heavily uses the tree of Self-Avoiding Walks construction by Weitz. This construction provides a good way of understanding spin correlations in Ising models, essentially by "unfolding" the graph into a tree. This representation is useful is because on trees, with zero external field, if $p$ is the unique path between $u$ and $v$, then
\begin{equation}\label{cov-on-trees}
\E(\sigma_u\sigma_v)=\PL{e\in p}\tanh(\beta_e).
\end{equation}
\begin{defn}
Let $G$ be a graph. On the vertices of $G$ fix an enumeration and let $v$ be one of the vertices of $G$. The tree of Self-Avoiding Walks $T_{\text{SAW}}(G,v)$ of $G$ rooted at $v$ is the tree of walks originating at $v$ that do not intersect themselves, except when the walk closes a cycle, at which point the walk ends. When that happens, the second copy of the twice-appearing vertex is fixed to be a $+$, if the edge closing the cycle is larger than the edge starting the cycle, and $-$ otherwise. In this case, comparing the edges really means comparing the respective vertices with respect to the enumeration we initially fixed.
\end{defn}
In the above construction, there is an implied map $\phi$, which maps every vertex of $T_{\text{SAW}}(G,v)$ to its label in the original graph $G$. The reason why this construction is useful, is the following Theorem, essentially proven in \cite{Weitz}.
\begin{thm}\label{Weitz}
Let $y\in V$, and consider the Ising model on $G$, conditioned on $\sigma_y=+$. Also, let $\sigma_{c,y}$ be the corresponding conditions on $T_{\text{SAW}}(G,v)$ obtained as follows: Every vertex $u$ on $T_{\text{SAW}}(G,v)$ which has $\phi(u)=y$ and has not been set to have a specific value from the construction of $T_{\text{SAW}}(G,v)$, is set to be a $+$. Then,
\begin{equation}\label{T-SAW-rel}
\PR_G(\sigma_v=+|\ \sigma_y=+)=\PR_{T_{\text{SAW}}(G,v)}(\sigma_\rho=+|\ \sigma_{c,y}).
\end{equation}
\end{thm}
Next, we prove a lemma that allows the bounding of correlations in a tree.
\begin{lem}\label{spin-corr}
Let $\mathcal{T}$ be a tree and let $(\sigma_v)_{v\in V(\mathcal{T})}$ be a configuration drawn from the Ising model on $\mathcal{T}$ with external field $h$. Denote by $\E_h$ the expectation under the Ising model with external field $h$. Then, if $\rho$ is the root of $\mathcal{T}$, $v_1,...,v_k$ are vertices of $\mathcal{T}$ that have degree at most $d'$ in $\mathcal{T}$ and $\E_h(\sigma_\rho)=0$, the following inequality holds:
$$\E_h\left(\sigma_\rho|\sigma_{v_1}=\cdots=\sigma_{v_k}=+\right)\le c(d')\cdot\SL{i=1}^k\E_0\left(\sigma_\rho|\sigma_{v_i}=+\right),$$
for some constant $c(d')$ that depends only on $d'$.
\end{lem}
In the proof of Lemma \ref{spin-corr}, we will make use of the following result from \cite{D-S-S}.
\begin{thm}[Corollary 1.3, \cite{D-S-S}]\label{D-S-S}
Let $g:V\to[-\infty,\infty]$ be an external field. Then, for all $u,v\in V$
\begin{equation}\label{0-max-cov}
\E_g(\sigma_u\sigma_v)-\E_g(\sigma_u)\E_g(\sigma_v)\le\E_0(\sigma_u\sigma_v).
\end{equation}
\end{thm}
\begin{proof}[Proof of Lemma \ref{spin-corr}]
In what follows, when the external field is omitted, it is assumed to be $h$. We can write
\begin{align*}
\E\left(\sigma_\rho|\sigma_{v_1}=\cdots=\sigma_{v_k}=+\right)&=\SL{i=1}^k\left[\E\left(\sigma_\rho|\sigma_{v_1}=\cdots=\sigma_{v_i}=+\right)-\E\left(\sigma_\rho|\sigma_{v_1}=\cdots=\sigma_{v_{i-1}}=+\right)\right]\\&=\SL{i=1}^k\left[\E_{h_{i-1}}\left(\sigma_\rho|\sigma_{v_i}=+\right)-\E_{h_{i-1}}\left(\sigma_\rho\right)\right],
\end{align*}
where $h_{i-1}$ is the new external field coming from conditioning on $\{\sigma_{v_1}=\cdots=\sigma_{v_{i-1}}=+\}$ as well. We now calculate and bound each term of the above sum separately. For convenience, let $p_i=\PR_{h_{i-1}}(\sigma_{v_i}=+), m_i^\pm=\E_{h_{i-1}}(\sigma_\rho|\sigma_{v_i}=\pm)$. Of course, $m_i^+\ge m_i^-$. Also, if $u_1,\dots,u_s$ are the neighbors of $v_i$, 
\begin{equation}\label{p_i-ineq}
p_i\ge\PR(v_i=+|\ u_1=\cdots=u_s=-)\ge\dfrac{1}{1+e^{2\beta d'}}.
\end{equation}
Since $\E_{h_{i-1}}(\sigma_\rho)=p_im_i^++(1-p_i)m_i^-$,
\begin{align*}
\E_{h_{i-1}}(\sigma_\rho|\sigma_{v_i}=+)-\E_{h_{i-1}}(\sigma_\rho)=(1-p_i)(m_i^+-m_i^-)
\end{align*}
and
\begin{align*}
\text{Cov}_{h_{i-1}}(\sigma_\rho,\sigma_{v_i})&=\E_{h_{i-1}}(\sigma_\rho\sigma_{v_i})-\E_{h_{i-1}}(\sigma_\rho)\E_{h_{i-1}}(\sigma_{v_i})\\&=p_im_i^+-(1-p_i)m_i^--(2p_i-1)(p_im_i^++(1-p_i)m_i^-)\\&=2p_i(1-p_i)(m_i^+-m_i^-).
\end{align*}
Therefore,
\begin{align*}
\E_{h_{i-1}}\left(\sigma_\rho|\sigma_{v_i}=+\right)-\E_{h_{i-1}}\left(\sigma_\rho\right)
&=\frac1{2p_i}\cdot\text{Cov}_{h_{i-1}}(\sigma_\rho,\sigma_{v_i})\\
&\le \frac1{2p_i}\cdot\text{Cov}_0(\sigma_\rho,\sigma_{v_i})=\frac1{2p_i}\cdot\E_0(\sigma_\rho|\sigma_{v_i}=+),
\end{align*}
where the inequality is by Theorem~\ref{D-S-S}.\ Due to equation (\ref{p_i-ineq}), the statement of the Lemma is true with $c(d')=\frac{1+e^{2\beta d'}}{2}$.
\end{proof}

\section{Proof of Theorem \ref{main-thm-determ}}
In this section, we prove Theorem \ref{main-thm-determ}. Throughout the proof, unless stated otherwise, we assume that for the graph $G$, conditions 1-4 of Theorem \ref{main-thm-determ} are satisfied and that $A,B$ and $H$  denote the sets defined therein.

\subsection{Polynomial Relaxation on $H$}

For vertices $u,v\in V$ that are connected, let $\beta_{u,v}^{(1)}=0$, if $u,v\in B$, or $\beta_{u,v}^{(1)}=\beta_c$ otherwise. The measure $\mu_{\beta,h}^{(1)}$ is defined to be the Ising Model on $G$ with these exact interactions, i.e. the measure on $\{-1,1\}^V$ which satisfies
$$\mu_{\beta,h}^{(1)}(\sigma)\varpropto\exp\left(\SL{u\sim v}\beta_{u,v}^{(1)}\sigma_u\sigma_v+\SL{v\in V}h(v)\sigma_v\right).$$
Let $(Y_t^{(1)})_{t\ge0}$ be the continuous time Glauber Dynamics for this measure and the chain $(\tilde{Y}_t^{(1)})_{t\ge0}$ be the projection of that chain on $B$. Also, for a configuration $\sigma_0\in\{-1,1\}^V$, let $(Y_t^{(2)})_{t\ge0}$ the Glauber Dynamics on $A$, for which $Y_0^{(2)}=\sigma_0$ and each vertex in $A$ is updated at rate 1, given the configuration $\sigma_{0,B}$. 
\begin{lem}\label{poly-relax}
The relaxation times of $(\tilde{Y}_t^{(1)})_{t\ge0}$ and $(Y_t^{(2)})_{t\ge0}$ are both polynomial, i.e. there exists a constant $D_2>0$ such that for any $\sigma_0\in\{-1,1\}^V$,
$$\text{gap}(\tilde{Y}_t^{(1)})\ \ \text{and}\ \ \text{gap}(Y_t^{(2)})\ge n^{-D_2}.$$
\end{lem}
In order to prove the lemma, we will use Proposition 4 of \cite{Mossel-Sly}, which we state below:
\begin{prop}[Proposition 4 of \cite{Mossel-Sly}]\label{prop-4-MS}
Let $G$ be a graph on $r$ vertices and $r+s-1$ edges that has a spanning tree $T$ which satisfies
$$m(T):=\max\limits_\gamma\SL{v\in\gamma}\deg(v)=m,$$
where $\gamma$ runs over all paths originating at the root of $T$. Then, the relaxation time of the Glauber Dynamics on $G$ with any boundary conditions and external fields satisfies
$$t_{\text{rel}}\le\exp\left(4\beta(m+s)\right).$$
\end{prop}

\begin{proof}[Proof of Lemma \ref{poly-relax}]
We use Proposition 4 of \cite{Mossel-Sly}. The graph $H$ may not be connected, but the Glauber Dynamics on $H$ is the product of the Glauber Dynamics on each one of the connected components. Therefore, it suffices to prove that the relaxation time is polynomial on each one of the components. However, because of conditions 1 and 2 in Theorem \ref{main-thm-determ}, we know that $m\le C\log(n)$ and $s\le1$ in each component. Therefore, Proposition \ref{prop-4-MS} implies that both $(Y_t^{(1)})_{t\ge0}$ and $(Y_t^{(2)})_{t\ge0}$ have polynomial relaxation times. As far as $(\tilde{Y}_t^{(1)})_{t\ge0}$ is concerned, we prove that
$$\text{gap}(\tilde{Y}_t^{(1)})\ge\text{gap}(Y_t^{(1)}).$$
Indeed, consider the function $g:\{-1,1\}^B\to\R$ which minimizes $\tilde{\mathcal{E}}^{(1)}(g,g)/\text{Var}(g)$ (this minimum is equal to the spectral gap of $(\tilde{Y}_t^{(1)})_{t\ge0}$). Then, the function $f:\{-1,1\}^V\to\R$ defined as $f(\sigma_A,\sigma_B)=g(\sigma_B)$ has
\begin{align*}
\mathcal{E}^{(1)}(f,f):&=\SL{\sigma\in\{-1,1\}^V}\SL{v\in V}\mu_\beta^{(1)}(\sigma)\cdot q^{(1)}(\sigma,\sigma^{\oplus v})\cdot(f(\sigma)-f(\sigma^{\oplus v}))^2
\\&=\SL{\sigma_B\in\{-1,1\}^B}\SL{v\in B}\ (g(\sigma_B)-g(\sigma_B^{\oplus v}))^2\SL{\sigma_A\in\{-1,1\}^A}\mu_\beta^{(1)}(\sigma)\cdot q^{(1)}(\sigma,\sigma^{\oplus v})\\&=\SL{\sigma_B\in\{-1,1\}^B}\SL{v\in B}\tilde{\mu}_\beta^{(1)}(\sigma_B)\cdot\tilde{q}^{(1)}(\sigma_B,\sigma_B^{\oplus v})\cdot(g(\sigma_B)-g(\sigma_B^{\oplus v}))^2
\\&=\tilde{\mathcal{E}}^{(1)}(g,g)
\end{align*}
and similarly,
$$\text{Var}_{\mu_\beta^{(1)}}(f)=\text{Var}_{\tilde{\mu}_\beta^{(1)}}(g).$$
Therefore, $f$ achieves the same ratio, $\mathcal{E}^{(1)}(f,f)/\text{Var}(f)=\tilde{\mathcal{E}}^{(1)}(g,g)/\text{Var}(g)$ which completes the proof.
\end{proof}
\subsection{Reduction to the Restricted Dynamics}
In this subsection, we consider a series of Markov chains, all defined in continuous time.
\begin{itemize}
\item 
$(X_t^{(1)})_{t\ge0}$ is the usual Glauber dynamics on the graph, in which every vertex has its spin updated at rate 1.
\item 
$(X_t^{(2)})_{t\ge0}$ is the accelerated dynamics, in which vertices in $A$ are updated at rate $n^{4D_2+4D_3+4}$, where $D_2$ is the exponent of Lemma \ref{poly-relax}, $D_3$ is the exponent of Lemma \ref{comp5-6} and vertices in $B$ are updated at rate 1.
\item 
$(X_t^{(3)})_{t\ge0}$ is the Markov Chain in which at rate $|B|$, there is an update of the spin at a uniformly random chosen vertex in $B$, and then there is a full update on the vertices of $A$ according to the Ising Model, given the configuration on $B$.
\item 
$(X_t^{(4)})_{t\ge0}$ is the projection of $(X_t^{(3)})_{t\ge0}$ on the set $B$.
\item 
$(X_t^{(5)})_{t\ge0}$ is the Restricted Dynamics Markov Chain on $B$, i.e. the continuous time Glauber Dynamics for the measure $\nu$ defined as
$$\nu(x_B)=\PR_{\mu_\beta}(\sigma_B=x_B),\ \text{for each}\ x_B\in\{-1,1\}^B.$$
\end{itemize}
Our next goal is to prove a series of comparisons between these chains that will be useful in the proof of Theorem \ref{main-thm-determ}.
\begin{lem}\label{comp1-2}
The spectral gaps of chains $(X_t^{(1)})_{t\ge0}$ and $(X_t^{(2)})_{t\ge0}$ satisfy
$$\text{gap}(X_t^{(1)})\ge n^{-4D_2-4D_3-4}\cdot\text{gap}(X_t^{(2)}).$$
\end{lem}
\begin{lem}\label{comp2-3}
Let $\sigma_0\in\{-1,1\}^V$ be the initial condition for $(X_t^{(2)})_{t\ge0}$. Consider the (random) initial condition $\sigma_0^*$ for $(X_t^{(3)})_{t\ge0}$, in which $\sigma_{0,B}^*=\sigma_{0,B}$ (i.e. $\sigma_0^*$ has the same spins at $B$ as $\sigma_0$) and $\sigma_{0,A}^*$ is drawn from the Ising model on $A$ given $\sigma_{0,B}$. Then, the heat kernels of the Markov Chains $(X_t^{(2)})_{t\ge0}$ and $(X_t^{(3)})_{t\ge0}$ satisfy
\begin{align}\label{TV-2-3}
\sup\limits_{\sigma_0\in\{-1,1\}^V}\lVert H_{n^{D_2+D_3+2}}^{(2)}(\sigma_0,\cdot)-H_{n^{D_2+D_3+2}}^{(3)}(\sigma_0^*,\cdot)\rVert_{\text{TV}}=o(1),
\end{align}
as $n\to\infty$.
\end{lem}
\begin{lem}\label{comp3-4}
The spectral gaps of $(X_t^{(3)})_{t\ge0}$ and $(X_t^{(4)})_{t\ge0}$ satisfy
$$\text{gap}(X_t^{(3)})\ge\dfrac{1}{2}\cdot\text{gap}(X_t^{(4)}).$$
\end{lem}
\begin{lem}\label{comp4-5}
The spectral gaps of $(X_t^{(4)})_{t\ge0}$ and $(X_t^{(5)})_{t\ge0}$ satisfy
$$\text{gap}(X_t^{(4)})\ge c_0'\cdot\text{gap}(X_t^{(5)}).$$
\end{lem}
Before we move on to the proofs of the lemmas, we explain how to derive Theorem \ref{main-thm-determ}, using Lemmas \ref{poly-relax}, \ref{comp1-2}, \ref{comp2-3}, \ref{comp3-4}, \ref{comp4-5} and \ref{comp5-6}. Note that Lemma \ref{comp5-6} has not been stated or proven yet.
\begin{proof}[Proof of Theorem \ref{main-thm-determ}]
At first, due to Lemmas \ref{comp3-4} and \ref{comp4-5} we know that
$$\text{gap}(X_t^{(3)})\ge\dfrac{c_0'}{2}\cdot\text{gap}(X_t^{(5)})$$
Therefore, due to Theorem \ref{rel-spec-mix} and Lemmas \ref{poly-relax} and \ref{comp5-6},
\begin{align*}
t_{\text{mix}}^{(3)}(1/8)&\le c_2\cdot n\cdot \text{gap}^{-1}(X_t^{(5)})
\\&\le c_2\cdot n^{D_3+1}\cdot\text{gap}^{-1}(\tilde{Y}_t^{(1)})
\\&\le n^{D_1+D_3+2}.
\end{align*}
Due to Lemma \ref{comp2-3}, if $n$ is large enough,
$$t_{\text{mix}}^{(2)}(1/4)\le t_{\text{mix}}^{(3)}(1/8)\le n^{D_2+D_3+2},$$
so again due to Theorem \ref{rel-spec-mix},
$$\text{gap}(X_t^{(2)})\ge n^{-D_2-D_3-2}.$$
Combining this with Lemma \ref{comp1-2}, we conclude that
$$\text{gap}(X_t^{(1)})\ge n^{-5D_2-5D_3-6},$$
which implies the desired result.
\end{proof}
Throughout the proofs of Lemmas \ref{comp1-2}-\ref{comp4-5}, $\sigma$ and $\tau$ denote configurations in $\{-1,1\}^V$. Similarly, $\sigma_A,\tau_A$ and $\sigma_B,\tau_B$ denote configurations in $\{-1,1\}^A,\{-1,1\}^B$, respectively. Whenever $\sigma_A$ and $\sigma_B$ are specified, we set $\sigma(v)$ to be either $\sigma_A(v)$ or $\sigma_B(v)$, depending on whether $v\in A$ or $B$.
\begin{proof}[Proof of Lemma \ref{comp1-2}]
The two Markov Chains have the same stationary distribution, which means that for any $f:\{-1,1\}^V\to\R$, we have the equality
$$\text{Var}_{\mu_\beta,1}(f)=\text{Var}_{\mu_\beta,2}(f).$$
If we set $n_v=1$ for $v\in B$ and $n_v=n^{4D_2+4D_3+4}$ for $v\in A$, the Dirichlet forms of Chains 1 and 2 are
\begin{align*}
\mathcal{E}^{(1)}(f,f)&=\dfrac{1}{2}\SL{v\in V}\SL{\sigma}\mu_\beta(\sigma)\cdot q^{(1)}(\sigma,\sigma^{\oplus v})\cdot(f(\sigma)-f(\sigma^{\oplus v}))^2\ \ \text{and}\\ \mathcal{E}^{(2)}(f,f)&=\dfrac{1}{2}\SL{v\in V}\SL{\sigma}\mu_\beta(\sigma)\cdot n_v\cdot q^{(1)}(\sigma,\sigma^{\oplus v})\cdot(f(\sigma)-f(\sigma^{\oplus v}))^2.
\end{align*}
From these equalities, it is clear that for every $f$,
$$\mathcal{E}^{(2)}(f,f)\le n^{4D_2+4D_3+4}\cdot \mathcal{E}^{(1)}(f,f),$$
and the lemma follows.
\end{proof}
\begin{proof}[Proof of Lemma \ref{comp2-3}]
For every $t\ge0$ and $\sigma\in\{-1,1\}^V$, there exists a coupling $\Gamma^{t,\sigma}$ between the measures $H_t^{(A)}(\sigma_A,\cdot|\sigma_B)$ (the heat kernel of the rate $n^{4D_2+4D_3+4}$-Dynamics on $A$ given $\sigma_B$) and $\mu_{\beta,A}(\cdot|\sigma_B)$ (the restricted Ising model on $A$, given the configuration in $B$), such that if $(\tau_A,\tau_A')$ is drawn from $\Gamma^{t,\sigma}$,
$$\PR(\tau_A\neq\tau_A')=\lVert H_t^{(A)}(\sigma_A,\cdot|\sigma_B)-\mu_{\beta,A}(\cdot|\sigma_B)\rVert_{\text{TV}}.$$
We couple the rate-1 processes on vertices of $B$ to update at the same times. Let $E_n$ be the event that the number $N$ of updates of the rate-1 processes until time $n^{D_2+D_3+2}$ satisfies $n^{D_2+D_3+2}\le N\le n^{D_2+D_3+4}-1$ and $F_n$ be the event that each two rings of the rate-1 processes are at least $n^{-2D_2-2D_3-2}$ time apart (considering the artificial rings at times 0 and $n^{D_2+D_3+2}$ as well). Note that both $E_n$ and $F_n$ hold with high probability. From now on, we condition on the times of these updates $t_1,t_2,...,t_N$, let $t_0=0$, $t_{N+1}=n^{D_2+D_3+2}$ and work on the intersection $E_n\cap F_n$. Since vertices in $A$ are updated at rate $n^{4D_2+4D_3+4}$, and due to Lemma \ref{poly-relax}, the spectral gap $\gamma^{(A)}$ of the dynamics on $A$ is at least $n^{3D_2+4D_3+4}$. Therefore, Theorem \ref{rel-spec-mix} applied for $\varepsilon=e^{-n^{D_2}}$ implies that
$$t_{\text{mix}}(e^{-n^{D_2}})\le c_2\cdot n^{-3D_2-4D_3-4}\cdot n^{D_2}\cdot n^2\le n^{-2D_2-3D_3-2}.$$
Below, we will use the fact that this inequality holds for all initial conditions. Let $\tilde{\tau}_{0,A}=\sigma_{0,A}$ and $\Gamma_0$ be a coupling between the measures $H_{t_1/2}^{(2)}(\sigma_0,\cdot)$ and $H_0^{(3)}(\sigma_0^*,\cdot)$ such that if $(\tau_0,\tau_0')$ is drawn from $\Gamma_0$, then
$$\PR(\tau_0\neq\tau_0')=\lVert H_{t_1/2}^{(2)}(\sigma_0,\cdot)-H_0^{(3)}(\sigma_0^*,\cdot)\rVert_{\text{TV}}\le e^{-n^{D_2}},$$
since on the event $F_n$, $t_1/2\ge n^{-2D_2-3D_3-2}$.
Assume that for some $k\ge0$ we have constructed a coupling $\Gamma_k$ of the measures $H_{(t_k+t_{k+1})/2}^{{(2)}}(\sigma_0,\cdot)$ and $H_{t_k}^{(3)}(\sigma_0^*,\cdot)$ so that if $(\tau_k,\tau_k')$ is drawn from $\Gamma_k$, then
$$\PR(\tau_k\neq\tau_k')\le (3k+1)\cdot e^{-n^{D_2}}.$$
Construct the coupling $\Gamma_{k+1}$ of the measures $H_{(t_{k+1}+t_{k+2})/2}^{(2)}(\sigma_0,\cdot)$ and $H_{t_{k+1}}^{(3)}(\sigma_0^*,\cdot)$ in the following way: On the event $\{\tau_k=\tau_k'=\sigma\}$, draw $\tilde{\tau}_{k+1,A}$ from $H_{(t_{k+1}-t_k)/2}^{(A)}(\tau_{k,A},\cdot|\tau_{k,B})$ so that 
\begin{align*}
\PR(\tilde{\tau}_{k+1}\neq\tau_k)&=\lVert H_{(t_{k+1}-t_k)/2}^{(A)}(\tau_{k,A},\cdot|\tau_{k,B})-H_{(t_{k+1}-t_k)/2}^{(A)}(\tilde{\tau}_{k,A},\cdot|\tau_{k,B})\rVert_{\text{TV}}
\\&\le2\cdot\sup\limits_{\sigma}\ \lVert H_{(t_{k+1}-t_k)/2}^{(A)}(\sigma_A,\cdot|\sigma_B)-\mu_{\beta,A}(\cdot|\sigma_B)\rVert_{\text{TV}}\le2\cdot e^{-n^{D_2}}.
\end{align*}
On the event that $\{\tilde{\tau}_{k+1,A}=\tau_{k,A}'=\sigma\}$, perform the same update on $B$, to obtain $\tau_{k+1,B}$ and $\tau_{k+1,B}'$. Finally, draw $(\tau_{k+1,A},\tau_{k+1,A}')$ from $\Gamma^{(t_{k+2}-t_{k+1})/2,(\sigma_A,\tau_{k+1,B})}$. If any of these events fail, couple arbitrarily. Under this construction,
\begin{align*}
\PR(\tau_{k+1}\neq\tau_{k+1}')&\le\PR(\tau_k\neq\tau_k')+\PR(\tilde{\tau}_{k+1,A}\neq\tau_k)+\sup\limits_{\sigma}\ \lVert H_{(t_{k+2}-t_{k+1})/2}^{(A)}(\sigma_A,\cdot|\sigma_B)-\mu_{\beta,A}(\cdot|\sigma_B)\rVert_{\text{TV}} \\&\le (3k+1)\cdot e^{-n^{D_2}}+2\cdot e^{-n^{D_2}}+e^{-n^{D_2}}=(3k+4)\cdot e^{-n^{D_2}},
\end{align*}
which is the desired property for $\Gamma_{k+1}$. In this way, we have successfully constructed a coupling $\Gamma_{N}$ between $H_{(t_N+n^{D_2+D_3+2})/2}^{(2)}(\sigma_0,\cdot)$ and $H_{t_N}^{(3)}(\sigma_0^*,\cdot)$. However, the measures $H_{t_N}^{(3)}$ and $H_{n^{D_2+D_3+2}}^{(3)}(\sigma_0^*,\cdot)$ are identical and by the analysis above, on the event $F_n$,
$$\lVert H_{n^{D_2+D_3+2}}^{(2)}(\sigma_0,\cdot)-H_{(t_N+n^{D_2+D_3+2})/2}^{(2)}(\sigma_0,\cdot)\rVert_{\text{TV}}\le2\cdot e^{-n^{D_2}}.$$
Therefore,
\begin{align*}
\lVert H_{n^{D_2+D_3+2}}^{(2)}(\sigma_0,\cdot)-H_{n^{D_2+D_3+2}}^{(3)}(\sigma_0,\cdot)\rVert_{\text{TV}} \le\ &\PR(E_n)+\PR(F_n)+3\cdot n^{D_2+D_3+4}\cdot e^{-n^{D_2}}=o(1),
\end{align*}
and the lemma follows.
\end{proof}
\begin{proof}[Proof of Lemma \ref{comp3-4}]
Let $Q^{(3)},Q^{(4)}$ be the transition rate matrices for $(X_t^{(3)})_{t\ge0}, (X_t^{(4)})_{t\ge0}$. We prove that every eigenvalue of $Q^{(3)}$ not equal to $-|B|$ is also an eigenvalue of $Q^{(4)}$. For each $\sigma,\tau\in\{-1,1\}^V$:
\begin{itemize}
\item
Assume $\tau_B=\sigma_B^{\oplus v}$. Then, the rate of transition $q^{(3)}(\sigma,\tau)$ is equal to
$$q^{(3)}(\sigma,\tau)=\dfrac{\mu(\sigma^{\oplus v})}{\mu(\sigma)+\mu(\sigma^{\oplus v})}\cdot\dfrac{\mu(\tau)}{\mu(\tau_B)}.$$
\item 
Assume $\tau_B=\sigma_B$ and $\tau_A\neq \sigma_A$. Then, the rate of transition $q(\sigma,\tau)$ equals
$$q^{(3)}(\sigma,\tau)=\SL{v\in B}\dfrac{\mu(\sigma)}{\mu(\sigma)+\mu(\sigma^{\oplus v})}\cdot\dfrac{\mu(\tau)}{\mu(\sigma_B)}.$$
\item 
Therefore, the transition rate $q^{(3)}(\sigma)$ is equal to
\begin{align*}
q^{(3)}(\sigma)&=-q^{(3)}(\sigma,\sigma)\\&=\SL{\tau_A}\SL{v\in B}\dfrac{\mu(\sigma^{\oplus v})}{\mu(\sigma)+\mu(\sigma^{\oplus v})}\cdot\dfrac{\mu(\tau)}{\mu(\tau_B)}+\SL{\tau_A\neq\sigma_A}\SL{v\in B}\dfrac{\mu(\sigma)}{\mu(\sigma)+\mu(\sigma^{\oplus v})}\cdot\dfrac{\mu(\tau)}{\mu(\sigma_B)}\\&=\SL{v\in B}\dfrac{\mu(\sigma^{\oplus v})}{\mu(\sigma)+\mu(\sigma^{\oplus v})}+\left(1-\dfrac{\mu(\sigma)}{\mu(\sigma_B)}\right)\SL{v\in B}\dfrac{\mu(\sigma)}{\mu(\sigma)+\mu(\sigma^{\oplus v})}\\&=|B|-\dfrac{\mu(\sigma)}{\mu(\sigma_B)}\SL{v\in B}\dfrac{\mu(\sigma)}{\mu(\sigma)+\mu(\sigma^{\oplus v})}.
\end{align*}
\end{itemize}
We analyze the transition rates of $Q^{(4)}$ in a similar way. For each $v\in B$,
\begin{align}\label{q4-rates}
q^{(4)}(\sigma_B,\sigma_B^{\oplus v})=\dfrac{1}{\mu(\sigma_B)}\SL{\sigma_A}\mu(\sigma)\cdot\dfrac{\mu(\sigma^{\oplus v})}{\mu(\sigma)+\mu(\sigma^{\oplus v})},
\end{align}
and
$$q^{(4)}(\sigma_B)=\SL{v\in B}q^{(4)}(\sigma_B,\sigma_B^{\oplus v}).$$
Let $\lambda\neq-|B|$ be an eigenvalue of the matrix $Q^{(3)}$, and assume $f:\{-1,1\}^V\to\R$ is an eigenfunction for this eigenvalue, with $f\neq0$. Then, for every $\sigma\in\{-1,1\}^V$,
\begin{align}\label{Q3-eigen}
\nonumber\SL{\tau}q^{(3)}(\sigma,\tau)f(\tau)&=\lambda f(\sigma)\\
\nonumber\Rightarrow \SL{v}\SL{\tau:\tau_B=\sigma_B^{\oplus v}}q^{(3)}(\sigma,\tau)f(\tau)+\SL{\tau:\tau_B=\sigma_B}q^{(3)}(\sigma,\tau)f(\tau)&=\lambda f(\sigma)\\
\Rightarrow \SL{v\in B}\dfrac{\mu(\sigma^{\oplus v})}{\mu(\sigma)+\mu(\sigma^{\oplus v})}\cdot g(\sigma_B^{\oplus v})+\SL{v\in B}\dfrac{\mu(\sigma)}{\mu(\sigma)+\mu(\sigma^{\oplus v})}\cdot g(\sigma_B)&=(\lambda+|B|)f(\sigma),
\end{align}
where we set $g:\{-1,1\}^B\to\R$ to be the function
$$g(\sigma_B)=\dfrac{1}{\mu(\sigma_B)}\SL{\tau:\tau_B=\sigma_B}\mu(\tau)f(\tau).$$
Therefore, since $f\neq0$ and $\lambda\neq-|B|$, we also have that $g\neq0$. For any $\sigma_B\in\{-1,1\}^B$, adding up relations (\ref{Q3-eigen}) for all $\sigma$ whose restriction on $B$ is $\sigma_B$ implies
\begin{align*}
(\lambda+|B|)\cdot g(\sigma_B)&=(\lambda+|B|)\cdot\dfrac{1}{\mu(\sigma_B)}\cdot\SL{\sigma_A}\mu(\sigma)f(\sigma)\\&=g(\sigma_B)\cdot\dfrac{1}{\mu(\sigma_B)}\SL{v\in B}\SL{\sigma_A}\dfrac{\mu(\sigma)^2}{\mu(\sigma)+\mu(\sigma^{\oplus v})}\\&+\dfrac{1}{\mu(\sigma_B)}\SL{v\in B}g(\sigma_B^{\oplus v})\cdot\SL{\sigma_A}\dfrac{\mu(\sigma)\mu(\sigma^{\oplus v})}{\mu(\sigma)+\mu(\sigma^{\oplus v})}.
\end{align*}
So, if $W$ is the matrix defined by
$$W(\sigma_B,\sigma_B^{\oplus v})=\dfrac{1}{\mu(\sigma_B)}\SL{\sigma_A}\dfrac{\mu(\sigma)\mu(\sigma^{\oplus v})}{\mu(\sigma)+\mu(\sigma^{\oplus v})}$$
and
$$W(\sigma_B,\sigma_B)=\dfrac{1}{\mu(\sigma_B)}\SL{v\in B}\SL{\sigma_A}\dfrac{\mu(\sigma)^2}{\mu(\sigma)+\mu(\sigma^{\oplus v})},$$
for all $\sigma_B\in\{-1,1\}^B$, the function $g$ must be an eigenfunction of $W$ for the eigenvalue $\lambda+|B|$. Moreover, observe that $W=Q^{(4)}+|B|\cdot I$. Therefore, $\lambda$ is also an eigenvalue of $Q^{(4)}$. As mentioned in Definition \ref{CTMC}, let $0=\lambda_1^{(3)}<\lambda_2^{(3)}\le\cdots\le\lambda_{2^n}^{(3)}$ and $0=\lambda_1^{(4)}<\lambda_2^{(4)}\le\cdots\le\lambda_{2^{|B|}}^{(4)}$ be the list of eigenvalues of $-Q^{(3)}$ and $-Q^{(4)}$, respectively. We consider two cases:
\begin{enumerate}
\item 
$\lambda_2^{(3)}<|B|$. Then, by our analysis above, $\lambda_2^{(3)}=\lambda_2^{(4)}$, so $\text{gap}(X_t^{(3)})=\text{gap}(X_t^{(4)})$. Lemma \ref{comp3-4} is proven in this case.
\item 
$\lambda_2^{(3)}\ge |B|$. Observe that for each $\sigma_B\in\{-1,1\}^B$,
\begin{align*}
\SL{\tau_B\in\{-1,1\}^B}|q^{(4)}(\sigma_B,\tau_B)|&=2\cdot\dfrac{1}{\mu(\sigma_B)}\cdot\SL{v\in B}\SL{\sigma_A}\dfrac{\mu(\sigma)\mu(\sigma^{\oplus v})}{\mu(\sigma)+\mu(\sigma^{\oplus v})}\\&\le2\cdot|B|\cdot\dfrac{1}{\mu(\sigma_B)}\cdot\SL{\sigma_A}\mu(\sigma)=2\cdot|B|,
\end{align*}
which implies that $\lambda_2^{(4)}\le2\cdot |B|\le2\cdot\lambda_2^{(3)}$. The result follows in this case as well.
\end{enumerate}
\end{proof}
\begin{proof}[Proof of Lemma \ref{comp4-5}]
Due to the property 4 in Theorem \ref{main-thm-determ}, the degree of each vertex in $B$ is at most $C'\cdot d$, so for each configuration $\sigma\in\{-1,1\}^V$ and $v\in B$,
$$\exp(-2\beta C'd)\le\dfrac{\mu(\sigma^{\oplus v})}{\mu(\sigma)}\le\exp(2\beta C'd)\ \ \Rightarrow\ \ \dfrac{\mu(\sigma^{\oplus v})}{\mu(\sigma)+\mu(\sigma^{\oplus v})}\ge\dfrac{1}{1+e^{2\beta C' d}}=:c_0'.$$
We already analyzed $q^{(4)}(\sigma_B,\sigma_B^{\oplus v})$ in (\ref{q4-rates}). As for $q^{(5)}(\sigma_B)$, from the definition of $X_t^{(5)}$ it is easy to see that
$$q^{(5)}(\sigma_B,\sigma_B^{\oplus v})=\dfrac{\nu(\sigma_B^{\oplus v})}{\nu(\sigma_B)+\nu(\sigma_B^{\oplus v})}=\dfrac{\SL{\sigma_A}\mu(\sigma^{\oplus v})}{\SL{\sigma_A}[\mu(\sigma)+\mu(\sigma^{\oplus v})]}\ge c_0'.$$
It is evident that the transition rates $q^{(4)}(\sigma_B,\sigma_B^{\oplus v})$ and $q^{(5)}(\sigma_B,\sigma_B^{\oplus v})$ satisfy
\begin{align*}
c_0'\le q^{(4)}(\sigma_B,\sigma_B^{\oplus v}),\ q^{(5)}(\sigma_B,\sigma_B^{\oplus v})\le1
\end{align*}
Comparing the Dirichlet forms of the two processes finishes the proof of this lemma.
\end{proof}
\subsection{Analysis of the Restricted Dynamics}
In this subsection, we prove the following lemma that will conclude the proof of Theorem \ref{main-thm-determ}.
\begin{lem}\label{comp5-6}
There exists a positive constant $D_4>0$ such that the spectral gaps of $(X_t^{(5)})_{t\ge0}$ and $(\tilde{Y}^{(1)})_{t\ge0}$ satisfy
$$\text{gap}(X_t^{(5)})\ge n^{-D_4}\cdot\text{gap}(\tilde{Y}_t^{(1)}).$$
\end{lem}
In the proof of Lemma \ref{comp5-6}, of crucial importance will be the analog of a lemma Chen and Eldan proved in \cite{StochLoc} for the spectral gap of the Glauber dynamics, instead of the modified log-Sobolev constant, which is the one analyzed in their paper. It is necessary to point out that Lemma \ref{Chen-Eldan-Spectral} is about the spectral gap of the discrete-time dynamics. In our case, we apply this Lemma for measures on $\{-1,1\}^{|B|}$. If $P$ is the transition matrix for the discrete-time Glauber Dynamics, since $Q=|B|\cdot(P-I)$, the spectral gap of the continuous-time dynamics differs from that of the discrete-time dynamics by a factor of $|B|$. Therefore, the statement below holds for the spectral gaps of the continuous-time dynamics as well.
\begin{lem}[Chen and Eldan, 2022]\label{Chen-Eldan-Spectral}
	
	Let $\nu$ be a measure on $\{-1,1\}^n$ and $J$ be a positive definite $n\times n$ matrix.\ For each $t\in[0,1]$ and $u\in\R^n$, define the measure $\nu_{t,u}$ to satisfy
	$$\dfrac{\text{d}\nu_{t,u}(x)}{\text{d}\nu(x)}\varpropto\exp\left(-t\langle x,Jx\rangle+\langle u,x\rangle\right).$$
	
	Assume that for all $t\in(0,1)$ and $u\in\R^n$ we have
	$$\lVert\text{Cov}(\nu_{t,u})\rVert_{\text{OP}}\le\alpha(t)\ \ \text{and}\ \ \text{gap}(P(\nu_{1,u}))\ge\varepsilon.$$

	Then,
	$$\text{gap}(P(\nu))\ge\varepsilon\cdot\exp\left(-2\lVert J\rVert_{\text{OP}}\int_0^1\alpha(t)\text{d}t\right).$$
\end{lem}
\begin{proof}[Proof of Lemma \ref{Chen-Eldan-Spectral}]
We repeat the proof of Theorem 49 in \cite{StochLoc}. All the lemmas and equations we refer to in this proof are from that paper. Let $(\nu_t)_{t\in[0,1]}$ be the process defined by the change of measure $\frac{\text{d}\nu_t(x)}{\text{d}\nu(x)}=F_t(x)$, where $F_t$ satisfies Equation (5) (of \cite{StochLoc}) with $C_t=(2J)^{1/2}$. Then, due to Fact 14,
$$\dfrac{\text{d}\nu_t(x)}{\text{d}\nu(x)}=\exp\left(Z_t-t\langle x,Jx\rangle+\langle y_t,x\rangle\right),$$
where $Z_t$ is a normalizing constant ensuring that $\nu_t$ is a probability measure and $(y_t)$ is an $\R^n$-valued process. Because of Equation (20) and Theorem 46,
$$\text{gap}(P(\nu))\ge\varepsilon\cdot\E\left[\exp\left(\int_0^1\lVert C_s^{1/2}\text{Cov}(\nu_s)C_s^{1/2}\rVert_{\text{OP}}\ \text{d}s\right)\right]^{-1}\ge\varepsilon\cdot\exp\left(-2\lVert J\rVert_{\text{OP}}\int_0^1\alpha(t)\ \text{d}t\right).$$
Lemma \ref{Chen-Eldan-Spectral} follows.
\end{proof}
\begin{proof}[Proof of Lemma \ref{comp5-6}]
Let $J$ be the adjacency matrix of the graph induced by the vertices in $B$. Then, $\lVert J\rVert_{\text{OP}}\le C'd$, due to condition 4 of Theorem \ref{main-thm-determ} and the well-known inequality
\begin{align}\label{OP-norm}
\lVert A\rVert_{\text{OP}}\le\sup\limits_{i\in[n]}\SL{j=1}^n|a_{ij}|,
\end{align}
 for a symmetric matrix $A\in\R^{n\times n}$. Then, the matrix $\tilde{J}=J+2C'dI_{|B|}$ is positive-definite, and the Ising model with interaction matrices $J,\tilde{J}$ are identical. We now bound $\lVert\text{Cov}(\nu_{t,u})\rVert_{\text{OP}}$, for any $t,u$, using inequality (\ref{OP-norm}) combined with Theorem \ref{Weitz} and Lemma \ref{spin-corr}. Note that, since all vertices of $B$ have degree at most $C'd$, the assumptions of Lemma \ref{spin-corr} hold: All the vertices that are conditioned to have a fixed spin in the tree of Self-Avoiding Walks are vertices of $B$. Moreover, observe that for any $t>0$, the measure $\nu_{t,u}$ is an Ising model that, on each edge which has both vertices in $B$, puts inverse temperature $\beta(1-t)$. So, keeping in mind (\ref{cov-on-trees}), we find that for some constant $C(d)>0$ that depends only on $d$,
\begin{align*}
\lVert\text{Cov}(\nu_{t,u})\rVert_{\text{OP}}\le\ &\sup\limits_{v\in B}\SL{y\in B}\text{Cov}_{\nu_{t,u}}(\sigma_v,\sigma_y)\\\overset{(\ref{0-max-cov})}{\le}&\sup\limits_{v\in B}\SL{y\in B}\E_{\nu_{t,0}}(\sigma_v\sigma_y)\\ \overset{(\ref{T-SAW-rel})}{\le}&\sup\limits_{v\in B}\SL{y\in B}\E_{T_{\text{SAW}}(G,v)}(\sigma_\rho|\sigma_{c,y})\\
\overset{\ref{spin-corr}}{\le} & C(d)\cdot\sup\limits_{v\in B}\SL{y\in B}\ \SL{y'\in T_{\text{SAW}}:\phi(y')=y}\E_0(\sigma_v|\sigma_{y'}=+)
\\\overset{(\ref{cov-on-trees})}{\le}&
C(d)\cdot\sup\limits_{v\in B}\SL{\ell=1}^n\theta^\ell\cdot\SL{p}\left(\dfrac{\tanh(\beta(1-t))}{\theta}\right)^{\#\text{good}(p)},
\end{align*}
where the sum is over all Self-Avoiding paths $p$ originating at $v$ and $\#\text{good}(p)$ indicates the number of edges of $p$ that connect vertices in $B$. Consider a vertex $v\in B$. Condition 4 of Theorem \ref{main-thm-determ} implies
\begin{align*}
&\SL{\ell=1}^n\theta^{\ell}\cdot\SL{p}\left(\dfrac{\tanh(\beta(1-t))}{\theta}\right)^{\#\text{good}(p)}\\\le\ &\SL{\ell=1}^{\delta\log_d(n)}d^{-\ell}S_v^{(\ell)}+\SL{\ell\ge\delta\log_d(n)}\theta^{0.4\ell}\cdot\tanh(\beta(1-t))^{0.6\ell}\cdot S_v^{(\ell)}\\\le \ & C'\delta\log_d(n)+\SL{\ell\ge\delta\log_d(n)}\theta^{0.4\ell}\cdot\tanh(\beta(1-t))^{0.6\ell}\cdot\sup\limits_{v\in B}S_v^{(\ell)}=:\alpha(t).
\end{align*}
Because of the fact that
$$\int_0^1\tanh(\beta(1-t))^{0.6\ell}\ \text{d}t=\int_0^\theta y^{0.6\ell}\cdot\dfrac{1}{\beta\cdot(1-y^2)}\ \text{d}y\le \dfrac{c}{\ell}\cdot\theta^{0.6\ell},$$
we get that
$$\int_0^1\alpha(t)\ \text{d}t\le C'\delta\log_d(n)+c\cdot\SL{\ell=1}^nd^{-\ell}\cdot\dfrac{1}{\ell}\cdot\sup\limits_{v\in B}S_v^{(\ell)}\le C'''\cdot\log(n),$$
which together with Lemma \ref{Chen-Eldan-Spectral} and Theorem \ref{rel-spec-mix} implies the desired result.
\end{proof}
\section{Structural Results for $G(n,d/n)$}
Before explaining the proof of Proposition \ref{ER-whp-prop}, we introduce the notion of an exploration process on a graph, which will be useful throughout this section.\\\\ 
Consider a vertex $v\in V$ and a subset $I\subseteq V$. The exploration process of the neighborhood of $v$ avoiding $I$ is the process that evolves as follows: At time $t=0$, vertex $v$ is active, vertices in $I$ are inactive and the rest of the vertices are neutral. At each time $t\ge1$, we choose the active vertex which has been active for the longest time (if more than one vertex satisfy this property, choose one of them arbitrarily). The neighborhood of the chosen vertex among all neutral vertices is explored. All neighbors of the chosen vertex become active, the chosen vertex becomes neutral and the rest of the neutral vertices remain neutral. The process ends when there is no active vertex. Often, we will stop this process at an earlier time, which we will specify in each different case.\\\\
In the case of an Erdös-Rényi graph $G(n,d/n)$, at time $t$, there is one vertex that becomes inactive from active and the number of neutral vertices that become active is $\text{Bin}(N_t,d/n)$, where $N_t$ is the number of neutral vertices at time $t$, which, conditional on the value of $N_t$, is independent of the rest of the process. Therefore, it is evident that the exploration process is stochastically dominated by the growth of a $\text{Bin}(n,d/n)$ Galton-Watson tree.\\\\
Occasionally in this section, we will be working under a no-tangle event $A_n$, a high-probability event describing certain characteristics of neighborhoods of $G(n,d/n)$.\\\\
For a vertex $v\in V$, consider the exploration process. Let $\tau_v,\tau_v'$ be the first times we have discovered sets of the form $B(v,r)$ (for some $r$ which we denote $r_v,r_v'$ from now on) which have $\ge n^{\delta/10}$ and $\ge n^{\delta/20}$ vertices, respectively. These times are $<\infty$ if, and only if, $v$ is in the large connected component. If $\tau_v=\tau_v'=\infty$, set $r_v=r_v=\infty$ and $B(v,r_v)=B(v,r_v')$ to be the connected component of $v$. Of course, if $K$ is large enough and $r_v,r_v'<\infty$, $r_v'\le r_v<K\log_d(n)$ for all $v$, with high probability. 
\begin{claim}
With high probability, for every $v\in V$, either $r_v=r_v'=\infty$ or $r_v'<r_v$.
\end{claim}
\begin{proof}
For each $v$, $r_v=r_v'<\infty$ implies that there exists a step of the process at which: the exploration process had previously discovered at most $n^{\delta/20}$ vertices, and at that step discovered at least $n^{\delta/10}-n^{\delta/20}$ points. This event has probability $e^{-\Omega(n^{\delta/10})}$, due to Chernoff's inequality. Taking a union bound over the possible values of $r_v$ and vertices $v$ implies the claim.
\end{proof}
For any graph $H$, we denote by $\text{tx}(H)=|E(H)|-|V(H)|+1$, i.e. the tree-excess of $H$. For some $\delta,c,c_0$ to be specified later, let $A_n$ be described by
$$A_n=A_n^{(1)}\cap A_n^{(2)}\cap A_n^{(3)}\cap A_n^{(4)},$$
where
\begin{align*}
A_n^{(1)}&=\left\{\sup\limits_{r\le\delta\log_d(n)}|B(v,r)|\cdot d^{-r}\le c\cdot\log_d(n),\ \forall v\in V\right\},\\
A_n^{(2)}&=\left\{\text{tx}(B(v,\delta\log_d(n))\le1,\ \forall v\in V\right\},
\\ A_n^{(3)}&=\left\{\text{tx}(B(v,r_v)\le1,\ \forall v\in V\right\}\ \ \text{and}\\
A_n^{(4)}&=\left\{|\partial B(v,r_v)|\ge c_0\cdot n^{\delta/10},\ \forall v\in V:\tau_v<\infty\right\}.
\end{align*}
We claim that for some $\delta,c,c_0$, the event $A_n$ is a high-probability event. This is true for $A_n^{(1)}$ and $A_n^{(2)}$, due to Lemma 2.5 of \cite{Mossel-Sly}. In Subsection \ref{4.3}, we prove that for some $c_0>0$ $A_n^{(3)}$ and $A_n^{(4)}$ are also high-probability events, verifying that $A_n$ is a high-probability event, for suitable values of $\delta,c,c_0$.
\subsection{Properties of the Partition}\label{4.1}
In this subsection, we prove that under the partition in Definition \ref{B-defn}, a graph drawn from $G(n,d/n)$ has properties 1-3 of Theorem \ref{main-thm-determ} with high probability. The property whose proof is the hardest is 3. In fact, as we will explain later in detail, the proof of property 3 contains the proof of property 2 as a special case, while property 1 will follow from property 3 and the fact that $A_n$ is a high-probability event.\\
\begin{prop}\label{almost-all-B}
With high probability in $G(n,d/n)$, for some $C>0$ sufficiently large, every Self-Avoiding Path $p$ of length $\ell\ge\delta\log_d(n)$, satisfies the inequality
$$\#\text{good}(p)\ge0.6\cdot\ell,$$
where $\#\text{good}(p)$ indicates the number of edges of $p$ connecting vertices in $B$.
\end{prop}
\begin{proof}
We prove that if the constant $C>0$ in Definition \ref{B-defn} is large enough, and for some small $\varepsilon$, with high probability, all paths $p$ of length $\ell=\delta\log_d(n)$ have at least 90\% of their vertices in $B$. Call a path $p$ \textit{bad}, if it does not have this property. Proving that there are no bad paths implies that at least 80\% of the edges of $p$ connect vertices in $B$, and by concatenation, for each path of length $\ge\delta\log_d(n)$ will have at least 60\% of edges connecting vertices in $B$. Call a path $p$ of length $\ell$ \textit{bad} if it does not have this property.\\\\ 
Let $p$ be a length $\ell$ Self-Avoiding path consisting of vertices $v_1,v_2,...,v_\ell$. We explore the $\delta\log_d(n)$-neighborhood of $p$: First, run the exploration process from $v_1$ avoiding $p$, stopped when all active vertices are at distance $\delta\log_d(n)$ (or when there are no active vertices, whichever happens first). Then, for each $i\in\{2,3,...,\ell\}$, run the exploration process avoiding $p$ and all vertices discovered in a previous process. Each process stops when all active vertices are at distance $\delta\log_d(n)$ of the respective $v_i$ (or, similarly, when there are no active vertices, depending on which happens first).\\\\ 
Let $B_{w,p}(k)$ be the set of vertices reachable from $w$ with a path of length $k$ not having an edge in $p$ discovered in the exploration process. For each $w\in p$, let $$Y_w=\max\limits_{r\le\delta\log_d(n)}|B_{w,p}(r)|\cdot d^{-r}\ \ \text{and}\ \ Y_w'=\max\limits_{r\le\delta\log_d(n)}Z_{w,r}\cdot d^{-r},$$ where $Z_{w,r}$ is the number of vertices of a Galton-Watson tree rooted at $w$ at generation $r$, with offspring distribution Bin$(n,d/n)$. Because of the way the exploration process is run,
\[
Y_w\preceq \max\limits_{r\le\delta\log_d(n)}d^{-r}\sum_{r'\leq r}Z_{w,r'} \leq (1-d^{-1})^{-1}\cdot Y_w',
\]
and all of the $Y_w'$ are independent.\\\\
Next, we reveal all the edges within the $\delta\log_d(n)$ neighborhood of $p$. On the event $A_n$, there is at most one pair $\{v_1,v_2\}$ of vertices in $p$ such that $B_{v_1,p}(\delta\log_d(n))\cap B_{v_2,p}(\delta\log_d(n))\neq\varnothing$. Also, due to Lemma 2.5 of \cite{Mossel-Sly}, $Y_{v_i}\le c_0\cdot\log_d(n)$. Assume $v\in p$ is such that $\text{d}_p(w,v_i)\ge\log\log(n)^2$. The number of vertices in the $r$-neighborhood of $v$ are at most
\begin{align*}
\SL{w\in p, \text{d}_p(w,v)\le r}|B_{w,p}(r-\text{d}_p(w,v))|&\le2c_0\log_d(n)+d^r\SL{w\in p, w\neq v_i}Y_w\cdot d^{-\text{d}_p(w,v)}\\&\le2c_0\log_d(n)+(1-d^{-1})^{-1}\cdot d^r\SL{w\in p, \neq v_i}Y_w'\cdot d^{-\text{d}_p(w,v)},
\end{align*}
if $r\ge\log\log(n)^2$ and $(1-d^{-1})^{-1}\cdot d^r\SL{w\in p, \neq v_i}Y_w'\cdot d^{-\text{d}_p(w,v)}$, otherwise. In any case, if $v\in A$ there exists some $r\le\delta\log(n)$ for which $|B_r(v)|>C\cdot d^r$, which implies
$$\SL{w\in p, w\neq v_i}Y_w'\cdot d^{-\text{d}_p(w,v)}>C-(1-d^{-1})^{-1}-1.$$
Now, having at least 5\% bad vertices among $\{v\in p:\text{d}_p(v,v_i)\ge\log\log(n)^2\}$ implies that
\begin{align*}
\SL{v\in p,\text{d}_p(v,v_i)\ge\log\log(n)^2}\SL{w\in p, w\neq v_i}Y_w'\cdot d^{-\text{d}_p(w,v)}&>0.05\cdot\left(\ell-4\log\log(n)^2\right)\cdot(C-(1-d^{-1})^{-1}-1)\\\Rightarrow\ 
\SL{w\in p, w\neq v_i}Y_w'&\ge\dfrac{d-1}{30d}\cdot\ell\cdot(C-(1-d^{-1})^{-1}-1).
\end{align*}
Notice also that the part we ignored in $p$ has size $O(\log\log(n)^2)$, making it impossible to have at least 10\% bad vertices in $p$ without having at least 5\% bad vertices in the part of $p$ we are studying. Using Lemma \ref{GW-study} which we prove below and large deviations techniques we find that for every $C'>0$,
$$\PR\left(\SL{w\in p, w\neq v_i}Y_w'\ge C'\cdot\ell\right)\le\exp\left(\left(-C'/2+\log(c_2)\right)\ell\right).$$
Putting
$$C':=\dfrac{d-1}{30d}\cdot(C-(1-d^{-1})^{-1}-1)$$
and making $C$ large enough, so that $C'/2-\log(c_2)>\dfrac{2\log(d)}{\delta}+1$,
we get that
$$\PR(p\ \text{is bad}\cap A_n|\ p\in G)\le\PR\left(\SL{w\in p, w\neq v_i}Y_w'\ge C'\cdot\ell\right)\le\dfrac{1}{n^2\cdot d^\ell}.$$
So, if $N_b$ is the number of bad paths present in $G$, on the event $A_n$ we get
\begin{align*}
\E\left(N_b\cdot\mathbf{1}_{A_n}\right)&=\SL{p}\PR(p\ \text{is bad}\cap p\in G\cap A_n)=\SL{p}\left(\dfrac{d}{n}\right)^\ell\cdot\PR(p\ \text{is bad}\cap A_n|\ p\in G)\\&\le\left(\dfrac{d}{n}\right)^\ell\cdot\dfrac{1}{n^2\cdot d^\ell}\cdot n^{\ell+1}\le\dfrac{1}{n}.
\end{align*}
The proposition follows.
\end{proof}
We turn to the study of the $Y_w'$.
\begin{lem}\label{GW-study}
For the random variables $Y_w'$, we have
$$\E(e^{Y_w'/2})\le4e^2:=c_2\in(0,\infty),$$
if $n$ is large enough.
\end{lem}
\begin{proof}
First of all, the process $(e^{Z_{w,r}\cdot d^{-r}})_r$ is a submartingale (due to it being the application of a convex function on a martingale), therefore Doob's maximal inequality holds. So,
$$\E(e^{Y_w'/2})\le4\cdot\E\left(e^{Z_{w,j}\cdot d^{-j}}\right),$$
for $j=\delta\log_d(n)$. However, we know that, conditioned on the value of $Z_{w,i-1}$, $Z_{w,i}$ has a binomial distribution with parameters $n\cdot Z_{w,i-1}$ and $d/n$. Therefore, for any $s\in(0,1)$,
\begin{align*}
\E\left(e^{sZ_{w,j}}|Z_{w,j-1}\right)&=\left(1+\dfrac{d}{n}\left(e^{s}-1\right)\right)^{n\cdot Z_{w,j-1}}\le\exp\left(d\cdot Z_{w,j-1}\cdot(e^{s}-1)\right)\\&\le\exp\left(ds(s+1)\cdot Z_{w,j-1}\right).
\end{align*}
Consider the sequence defined as $a_0=s$ and for $1\le i\le j$, $a_i=da_{i-1}(a_{i-1}+1)$. The above inequality immediately implies
$$\E\left(e^{sZ_{w,j}}\right)=\E\left(e^{a_0Z_{w,j}}\right)\le\E\left(e^{a_1Z_{w,j-1}}\right)\le\cdots\le e^{a_j}.$$
We claim that $a_j\le\dfrac{d^j\cdot s}{1-\frac{s}{d(d-1)}}$. This will imply the desired result for $s=d^{-j}$. Indeed, from the recursive relation we can see that $a_i>da_{i-1}$ for all $i$, so for each $i$,
$$\dfrac{1}{da_{i-1}}=\dfrac{1}{a_i}+\dfrac{a_{i-1}}{a_i}<\dfrac{1}{a_i}+\dfrac{1}{d}.$$
Using all of these inequalities for $1\le i\le j$, we get that
\begin{align*}
&\dfrac{1}{a_j}-\dfrac{1}{d^ja_0}=\SL{i=1}^jd^{-(j-i)}\left(\dfrac{1}{a_i}-\dfrac{1}{da_{i-1}}\right)\ge-\dfrac{1}{d}\SL{i=1}^jd^{-(j-i)}\ge-\dfrac{d^{-j}}{d(d-1)}\\
\Rightarrow\ &\dfrac{1}{a_j}\ge d^{-j}\left(\dfrac{1}{s}-\dfrac{1}{d(d-1)}\right),
\end{align*}
which concludes the proof of the lemma.
\end{proof}
\begin{cor}
With high probability, the connected components of $H$ have tree excess at most 1 and
$$\max\limits_\gamma\SL{v\in \gamma}\deg(v)\le C\log(n),$$
where the maximum runs over all Self-Avoiding paths $\gamma$ in $H$.
\end{cor}
\begin{proof}
At first, due to Proposition \ref{almost-all-B}, every connected component of the graph that is obtained from $G$ after deletion of the good edges is included in a neighborhood of radius at most $\delta\log_d(n)$ (as longer paths not connecting vertices in $B$ do not exist). Therefore, under the no-tangle event $A_n$, the tree excess of each connected component is at most 1. We turn to the second statement. It suffices to prove that for some $C>0$, with high probability, there exists no path $p$ of length $\ell\le\delta\log_d(n)$ whose vertices have sums of degrees $\ge C\log(n)$. Indeed, under the no-tangle event, each vertex has degree in the path at most 3 and as far as the degrees outside of the path, their sum is stochastically dominated by a $\text{Bin}(n\cdot\ell,d/n)$. Due to Chernoff's inequality, if $C>0$ is large enough,
\begin{align*}
\PR\left(\SL{v\in p}\deg(v)\ge C\log(n)\right)&\le\PR\left(\text{Bin}(n\cdot\ell,d/n)>C\log(n)/2\right)\\&\le\left(\dfrac{2e}{C}\right)^{C\log(n)/2}\\&\le n^{-2}\cdot d^{-\delta\log_d(n)}.
\end{align*}
Taking the expected number of paths satisfying this property will now finish the proof of the corollary, in the same way as at the end of the proof of Proposition \ref{almost-all-B}.
\end{proof}
\subsection{Non-Backtracking Walks in $G(n,d/n)$}\label{4.2}
Instead of bounding the number of Self-Avoiding Walks, we bound the number of walks that are initially (for the first $L=K\log_d(n)$ steps) Non-Backtracking and after that are Self-Avoiding. The reason for this is that we control the number of non-backtracking walks using a matrix representation and that this loses very little since initially Self-Avoiding Walks are almost equivalent to non-backtracking ones. Call $W$ the NB matrix of the graph, i.e. the matrix indexed by the directed edges of the graph, whose formula is
$$W_{ef}=\mathbf{1}_{e_2=f_1, e\neq f^{-1}},$$
for directed edges $e,f$. An important observation is that for each $k$, $(W^k)_{ef}$ is the number of Non-Backtracking Walks of length $k+1$ that start from $e$ and end at $f$. For two vertices $x,y\in V$ and a positive integer $k$, let $N_{xy}^{(k)}$ be the number of Non-Backtracking Walks from $x$ to $y$.
\begin{lem}\label{mixing-NB}
Let $v$ be the left Perron-Frobenius eigenvector of $W$. Then,
$$\dfrac{v(y)}{\lVert v\rVert_1}\cdot N_x^{(L)}-d^{(K+10)\log_d(n)/2}\le N_{xy}^{(L)}\le \dfrac{v(y)}{\lVert v\rVert_1}\cdot N_x^{(L)}+d^{(K+10)\log_d(n)/2}$$
for every $x,y\in V$, with high probability, where for a vertex $y\in V$ we set
$$v(y):=\SL{f:f_2=v}v(f).$$
\end{lem}
\begin{proof}[Proof of Lemma \ref{mixing-NB}]
Because of Propositions 10 and 11 in \cite{NBWalks}, we know that
\begin{align}\label{W-decomp-1}
W^\ell=\theta\zeta\check\varphi^*+J_\ell,
\end{align}
where:
\begin{itemize}
\item
$\ell\sim\varepsilon\log_d(n)$, for some $\varepsilon\in(0,1/6)$.
\item 
$\chi$ is the vector with $\chi(e)=1$ for each directed edge $e$, $\varphi=\dfrac{W^\ell\chi}{\lVert W^\ell\chi\rVert}$ and $\check\varphi$ is the vector for which $\check{\varphi}(e)=\varphi(e^{-1})$, for every $e$.
\item 
$\theta=\lVert W^\ell\check{\varphi}\rVert\in (c_0^{-1}\cdot d^\ell,c_0\cdot d^\ell)$ for some $c_0>0$.
\item 
$\zeta=W^\ell\check{\varphi}/\theta$.
\item 
$\langle\zeta,\check\varphi\rangle\ge c_0$ and $\lVert J_\ell\rVert\le\log(n)^c\cdot d^{\ell/2}$.
\end{itemize}
We also make use of Proposition 7 from \cite{NBWalks}.
\begin{prop}[Proposition 7 of \cite{NBWalks}]\label{prop7}
Let $A\in\mathcal{M}_n(\R)$, $\ell,\ell'$ be co-prime integers, $\lambda\neq0$ and $c_0,c_1>0$ such that: For each $k\in\{\ell,\ell'\}$ and some $x_k,y_k\in\R^n$, $J_k\in\mathcal{M}_n(\R),$
$$A^k=\lambda^kx_ky_k^*+J_k,\ \  \langle x_k,y_k\rangle\ge c_0\ \ \text{and}\ \ \lVert J_k\rVert<\dfrac{c_0^2}{2(\ell\vee\ell')c_1}\lambda^k.$$
Then, there exists a unit eigenvector $\psi$ for $\lambda_1(A)$ such that
$$\left\lVert\psi-\dfrac{x_\ell}{\lVert x_\ell\rVert}\right\rVert\le8c_0^{-1}\lVert J_\ell\rVert\cdot\lambda^{-\ell}.$$
\end{prop}
We apply Proposition \ref{prop7} to $W$ and $W^*$, for the decompositions (\ref{W-decomp-1}) and
$$(W^*)^\ell=\theta\check{\varphi}\zeta^*+J_\ell^*,$$
which follows directly from (\ref{W-decomp-1}). Also, we can write $W=\lambda_1uv^*+R$, where the vectors $u,v$ are the right and left (Perron-Frobenius) eigenvectors of $W$ for $\lambda_1$, normalized so that $\lVert u\rVert=\langle u,v\rangle=1$. With this decomposition of $W$, $v^*R=Ru=0$, which means that for any $\ell$,
\begin{align}\label{B-decomp-2}
W^\ell=\lambda_1^\ell uv^*+R^\ell.
\end{align}
Due to Proposition \ref{prop7},
\begin{align}
\lVert u-\zeta\rVert\ \ \text{and}\ \ \left\lVert\dfrac{v}{\lVert v\rVert}-\check\varphi\right\rVert = O(\log(n)^c\cdot d^{-\ell/2}).
\end{align}
From (\ref{W-decomp-1}) we can deduce that $J_\ell\check\varphi=0$, so
$$\theta\langle\zeta,\check\varphi\rangle\zeta=(W^\ell-J_\ell)\zeta=\lambda_1^\ell\cdot\zeta+W^\ell(\zeta-u)+\lambda_1^\ell(u-\zeta)-J_\ell\zeta.$$
Note that $\lVert W^\ell\rVert\le\theta+\lVert J_\ell\rVert=O(d^\ell)$, so taking norms in this last equation implies
$$\langle\zeta,\check\varphi\rangle=\dfrac{\lambda_1^\ell}{\theta}+O\left(\log(n)^c\cdot d^{-\ell/2}\right),$$
and 
$$\left\langle u,\dfrac{v}{\lVert v\rVert}\right\rangle=\langle\zeta,\check\varphi\rangle+\left\langle u-\zeta,\dfrac{v}{\lVert v\rVert}\right\rangle+\left\langle \zeta,\dfrac{v}{\lVert v\rVert}-\check\varphi\right\rangle=\dfrac{\lambda_1^\ell}{\theta}+O\left(\log(n)^c\cdot d^{-\ell/2}\right).$$ 
Therefore,
$$\lVert v\rVert=\dfrac{\theta}{\lambda_1^\ell}+O\left(\log(n)^c\cdot d^{-\ell/2}\right).$$
We are now ready to combine the two expressions for $W^\ell$ in order to bound the norm of $R^\ell$. Indeed:
\begin{align*}
\lVert R^\ell\rVert&\le\lVert J_\ell\rVert+\lVert J_\ell-R^\ell\rVert\\&\le
\lVert J_\ell\rVert+\lVert \theta\zeta\check\varphi^*-\lambda_1^\ell uv^*\rVert\\&\le
\lVert J_\ell\rVert+\theta\lVert \zeta\rVert\cdot\left\lVert \check\varphi-\dfrac{v}{\lVert v\rVert}\right\rVert+\theta\left\lVert\dfrac{v}{\lVert v\rVert}\right\rVert\cdot\lVert u-\zeta\rVert+\left|\dfrac{\theta}{\lVert v\rVert}-\lambda_1^\ell\right|\cdot\lVert uv^*\rVert\\&=O\left(\log(n)^c\cdot d^{\ell/2}\right).
\end{align*}
With this estimate, by raising (\ref{B-decomp-2}) to the power $K/\varepsilon$ (which may be large but is constant), we ensure that
$$W^{K\log_d(n)}=\lambda_1^{K\log_d(n)}uv^*+R^{K\log_d(n)},$$
where $\lVert R^{K\log_d(n)}\rVert=O\left(\log(n)^c\cdot d^{K\log_d(n)/2}\right)$. Analyzing the last relation tells us that
\begin{align}\label{no-of-NB}
N_{ef}^{(L)}=\lambda_1^L\cdot u(e)\cdot v(f)+(R^L)_{ef}
\end{align}
for directed edges $e,f$. Adding these relations over edges $e,f$ such that $e_1=x, f_2=y$ yields the result.
\end{proof}

We will need the fact that this error appearing in Lemma \ref{mixing-NB} is small enough for our purposes. Indeed, we will prove the following Lemma (which we will use at the very end of our proof):
\begin{lem}\label{small-error}
With high probability, if $K$ is large enough,
$$\SL{y}S_y^{(\ell-L)}\le d^{\ell-2L/3}.$$
\end{lem}

\begin{proof}[Proof of Lemma \ref{small-error}]
For each $w$, there are at most $n^{\ell-L}$ Self-Avoiding paths in the complete graph, and each one of them has probability $(d/n)^{\ell-L}$ of appearing in $G$. So,
$$\E\left[\SL{y}S_y^{(\ell-L)}\right]\le n\cdot n^{\ell-L}\cdot\left(\dfrac{d}{n}\right)^{\ell-L}=n\cdot d^{\ell-L}.$$
By Markov's Inequality,
$$\PR\left(\SL{y}S_y^{(\ell-L)}>d^{\ell-2L/3}\right)\le n\cdot d^{-L/3}=n^{1-\frac{K}{3}},$$
which is $o(1)$ if $K>3$.
\end{proof}
\subsection{Neighborhood Structure}\label{4.3}
We begin this subsection by proving that indeed, $A_n$ is a high-probability event.
\begin{lem}\label{tx-1}
With high probability, for all vertices $v\in V$, $\text{tx}(B(v,r_v))\le1$.
\end{lem}
\begin{lem}\label{expansion}
With high probability, for all $v\in V$ such that $r_v<\infty$, $|\partial B(v,r_v)|\ge c_0\cdot n^{\delta/10}$, for some $c_0>0$.
\end{lem}

\begin{proof}[Proof of Lemma \ref{tx-1}]
First of all, observe that with high probability, $|B(v,r_v)|\le 2dn^{\delta/10}$ for all $v\in G$. Indeed, let $v\in G$. Since $|B(v,r_v-1)|\le n^{\delta/10}$, the $|\partial B(v,r_v)|$ is stochastically dominated by a $\text{Bin}(n^{1+\delta/10},d/n)$, whose probability of exceeding $(2d-1)n^{\delta/10}$ is at most $e^{-\Omega(n^{\delta/10})}$. Therefore, this fails for all $v$ with high probability. Consider the exploration process from $v$. Due to the way the exploration process is run, the tree excess of the neighborhood we have explored is stochastically dominated by a $\text{Bin}(|B(v,r_v)|^2,d/n)$. So,
$$\PR\left(B(v,r_v)\ \text{has tree-excess} \ge 2\right)\le\PR\left(\text{Bin}(|B(v,r_v)|^2,d/n)\ge2\right)\le(d\cdot n^{-1+\delta})^2\le n^{-3/2}.$$
Taking a union bound over $v\in G$ concludes the proof.
\end{proof}
\begin{proof}[Proof of Lemma \ref{expansion}]
For each $v\in V$ with $r_v<\infty$, let $r_v''$ be the least radius such that $|\partial B(v,r_v'')|\ge n^{\delta/30}$. Because $r_v'\le K\log_d(n)$, we can see that $r_v''\le r_v'<r_v$. Next, we claim that with high probability, for each $v\in G$ and $r_v''+1\le k\le r_v$,
$$|\partial B(v,k)|\ge|\partial B(v,k-1)|\cdot\dfrac{d+1}{2}.$$
As analyzed before, the event that for every $v$, until radius $r_v$ we have discovered at most $n^{\delta/2}$ vertices is a high-probability event. So, the number of vertices the exploration process discovers at each step stochastically dominates a $\text{Bin}(|\partial B(v,k-1)|\cdot(n-n^{\delta/2}),d/n)$. Therefore,
\begin{align*}&\PR\left(|\partial B(v,k)|<|\partial B(v,k-1)|\cdot\dfrac{d+1}{2}\right)\\\le\ &\PR\left(\text{Bin}\left(|\partial B(v,k-1)|\cdot(n-n^{\delta/2}),\dfrac{d}{n}\right)<|\partial B(v,k-1)|\cdot\dfrac{d+1}{2}\right)\le e^{-\Omega(n^{\delta/30})}.
\end{align*}
This proves the claim. In order to finish the proof of this Lemma, we observe that with high probability,
\begin{align*}
n^{\delta/10}&\le|B(v,r_v)|\le\SL{k=r_v''}^{r_v}|\partial B(v,k)|+|B(v,r_v''-1)|\\&\le|\partial B(v,r_v)|\cdot\SL{k=0}^\infty\left(\dfrac{2}{d+1}\right)^k+K\log_d(n)\cdot n^{\delta/30}.
\end{align*}
This concludes the proof of the Lemma.
\end{proof}
We claim that even when conditioning on the existence of a specific not-so-long Self-Avoiding Path, $A_n$ remains a high-probability event. This will be used to give a conditional moment bound in Subsection \ref{4.4}.
\begin{lem}\label{cond-properties}
Let $\ell\le\sqrt{n}/e^{\sqrt{\log(n)}}$ and consider distinct vertices $v_1,...,v_\ell$. Let $p$ be the path $v_1,v_2,\dots,v_\ell$ and $E_p$ be the event that $p$ is present in the graph $G$.
The no-tangle event $A_n$ satisfies
$$\PR\left(A_n|\ E_p\right)=1-o(1),$$
as $n\to\infty$.
\end{lem}

Before we prove Lemma \ref{cond-properties}, we prove some claims that will be helpful.
\begin{claim}\label{claim-cond}
The following hold:
\begin{enumerate}
\item 
Conditional on $E_p$, with high probability, the $2\delta\log_d(n)$-neighborhood of $p$ does not contain cycles of length $\le2\delta\log_d(n)$.
\item
$\PR(|B(v,r_v)|<n^{\delta/2}\ \text{for all}\ v\in V|\ E_p)=1-o(1).$
\item 
Let $v\in V$. Conditional on $E_p$, we say that a vertex $v_i\in p$ was discovered naturally while exploring the neighborhood of $v$, if $\text{d}(v,v_i)\le\min(\text{d}(v,v_{i-1}),\text{d}(v,v_{i+1}))$. With high probability the following event holds: For every $v\in V$ and any pair $(v_i,v_j)$ such that $|i-j|\le\log(n)^2$, while exploring $B(v,r_v)$, at most one of the $v_i,v_j$ was discovered naturally.
\end{enumerate}
\end{claim}
\begin{proof}
\begin{enumerate}	
\item
From each $v_i$, we run the exploration process avoiding $p$ and all previously discovered vertices, until we have revealed the $2\delta\log_d(n)$ neighborhood of $p$.
Also, due to the proof of Lemma \ref{GW-study} and the inequality $x^2<2e^x$ for $x>0$,
\begin{align*}
\E(|\tilde{B}_p(v_i,2\delta\log_d(n))|^2)\le c\cdot d^{4\delta\log_d(n)}=c\cdot n^{4\delta},
\end{align*}
for every $v_i$. Observe that the only ways to create cycles of length $\le2\delta\log_d(n)$ are: 
\begin{itemize}
\item
A pair of vertices discovered at the same exploration process to connect via an unexplored edge.
\item 
A vertex in $\tilde{B}_p(v_i,2\delta\log_d(n))$ connects to some $v_j$, for which $|i-j|\le2\delta\log(n)$.
\end{itemize}
We now explain why both of these events do not happen with high probability. Let $N_i^{(1)}$ be the number of edges closing cycles in $\tilde{B}_p(v_i,2\delta\log_d(n))$. Our goal is to prove that $\SL{i=1}^\ell N_i^{(1)}=0$ with high probability. Of course, $N_i^{(1)}$ is stochastically dominated by a Bin$\left(|\tilde{B}_p(v_i,2\delta\log_d(n))|^2,d/n\right)$, so
$$\E\left(\SL{i=1}^\ell N_i^{(1)}\right)\le\SL{i=1}^\ell\dfrac{d}{n}\cdot\E\left(|\tilde{B}_p(v_i,2\delta\log(n))|^2\right)\le c\cdot n^{1/2-4\delta}=o(1),$$
if $\delta<1/8$. Also, let $N_i^{(2)}$ be the number of edges between a vertex in $\tilde{B}_p(v_i,2\delta\log_d(n))$ and a vertex $v_j$, for which $|i-j|\le2\delta\log_d(n)$. Similarly as before, $N_i^{(2)}$ is stochastically dominated by a Bin$\left(|\tilde{B}_p(v_i,2\delta\log_d(n))|\cdot4\delta\log_d(n),d/n\right)$, so
$$\E\left(\SL{i=1}^\ell N_i^{(2)}\right)\le\dfrac{2d\delta\log_d(n)}{n}\SL{i=1}^\ell\E\left(|\tilde{B}_p(v_i,2\delta\log_d(n))|\right)\le c\cdot\log_d(n)\cdot n^{-1/2+2\delta}=o(1),$$
if $\delta<1/4$.
\item 
Due to the definition of $r_v$, $|B(v,r_v-1)|<n^{\delta/10}$. Conditional on $E_p$, we prove that if $|B(v,r)|<n^{\delta/10}$, then $|B(v,r+1)|<n^{\delta/2}$ with probability $\ge1-e^{-\Omega(n^{\delta/2})}$. Indeed, $\partial B(v,r)$ has two types of vertices: The ones that are imposed by the existence of $p$ and the ones that are discovered using the exploration process. The number of vertices of the first type is in total at most $2\cdot n^{\delta/10}$, while the number of vertices of the second type is stochastically dominated by a Bin$(n^{\delta/10},d/n)$. Therefore,
\begin{align*}
&\PR\left(|B(v,r+1)|>n^{\delta/2}|\ p\in G,|B(v,r)|\le n^{\delta/10}\right)\\\le\ &\PR\left(\text{Bin}(n^{\delta/10},d/n\right)\ge n^{\delta/2})=e^{-\Omega(n^{\delta/2})}.
\end{align*}
Taking a union bound over $v\in V$ and the possible values of $r_v$ concludes the proof.
\item 
Observe that conditional on $E_p$, a vertex from the $v_i$ is more unlikely to be discovered naturally, and also that for conditional on $\{|B(v,r_v)|\le n^{\delta/2}\}$, for every two vertices, the event of them being discovered naturally is at most at likely as for 2 vertices to both be chosen in a list of $n^{\delta/2}$ vertices. So, for any $i,j$, during the exploration process of $B(v,r_v)$,
$$\PR(v_i\ \text{and}\ v_j\ \text{are discovered naturally}\ |\ E_p)\le(n^{-1+\delta/2})^2=n^{-2+\delta}.$$
Since there are $n$ possible choices for the initial vertex $v$ and at most $\ell\cdot\log(n)^2\le\sqrt{n}$ pairs of $(v_i,v_j)$ such that $|i-j|\le\log(n)^2$, taking a union bound proves the result.
\end{enumerate}
\end{proof}
We are ready to prove Lemma \ref{cond-properties}.
\begin{proof}[Proof of Lemma \ref{cond-properties}]
First, observe that 
$$\PR\left(|\partial B(v,r_v)|\ge c_0\cdot n^{\delta/10}\ \forall v\in V:\tau_v<\infty\ |\ E_p\right)=1-o(1).$$
The reason is contained in the proof of Lemma \ref{expansion} we provide above. Indeed, the exploration process at each step still stochastically dominates a Bin$(|\partial B(v,k-1)|\cdot(n-n^{\delta/2}),d/n)$ because when we condition on $E_p$, the number of children at generation $k$ can only increase. Therefore, all the steps of this proof work in the exact same way.\\\\
Next, we prove that conditional on $E_p$, with high probability, $|B_r(v)|\le c'\cdot d^r\cdot\log(n)$, for every $v\in V, r\le\delta\log_d(n)$, for some constant $c'$. Fix a vertex $v$ and run the exploration process from $v$, avoiding $p$ until $B_p\left(v,\delta\log_d(n)\right)$ has been explored. Due to Lemma \ref{GW-study},
\begin{align*}
\PR\left(\max\limits_{r\le\delta\log_d(n)}|B_p(v,r)|\cdot d^{-r}>c\cdot(1-d^{-1})^{-1}\log(n)\right)\le\PR(Y_v'>c\log(n))\le c_2\cdot e^{-c\log(n)/2}.
\end{align*}
Taking $c$ to be large enough makes this probability $o(n^{-1})$. Next, run the exploration process from each $v_i$ away from $p$ and $B_p(v,\delta\log_d(n))$, until reaching radius $\delta\log_d(n)$ from every $v_i$. Similarly as above, each neighborhood discovered has, with probability $1-o(1)$, the property that $|\tilde{B}_p(v_i,r)|\le c(1-d^{-1})^{-1}\cdot d^r\cdot\log(n)$. Moreover, due to the way the exploration process has ran, connections between neighborhoods can happen only if some $v_i$ connects to some vertex of the neighborhood of $v_j$, with $j<i$. However, for each $i$,
\begin{align*}
&\PR\left[\tilde{B}_p(v_i,\delta\log_d(n))\ \text{connects to}\ \ge2\ \text{other}\ \tilde{B}_p(v_j,\delta\log_d(n))\right]\\\le\ &\PR\left(\text{Bin}(n^{2\delta}\cdot\log(n)^4\cdot\ell,d/n)\ge2\right)\le n^{-3/2+4\delta}.
\end{align*}
Therefore, conditional on $E_p$, with high probability, each ball $B_{v,\delta\log_d(n)}(v_i,r)$ has at most $2c(1-d^{-1})^{-2}\cdot d^r\cdot\log(n)$. Finally, because of the way the exploration process was ran, the only way to discover vertices in $B(v,r)$ for $r\le\delta\log_d(n)$ that are not in $B_p(v,\delta\log_d(n))$ is if $B_p(v,\delta\log_d(n)-1)$ connects to some $v_i$. However, arguing similarly as before, with probability $1-o(n^{-1})$, $B_p(v,\delta\log_d(n)-1)$ connects to at most 1 of the $v_i$'s and on this event, there exists $v_i$ for which
$$B(v,r)\subseteq B_p(v,r)\cup B_{v,\delta\log_d(n)}(v_i,r).$$
This means that $|B(v,r)|\le c(1-d^{-1})^{-1}(1+2(1-d^{-1})^{-1})\cdot d^r\cdot\log(n)$ happens with probability $1-o(n^{-1})$. Taking a union bound proves that with high probability, conditional on $p$ being present, $|B(v,r)|\le c'\cdot d^r\cdot\log(n)$ for each $v\in V$ and $r\le\delta\log_d(n)$.\\\\
Next, we aim to prove that conditional on $E_p$, with high probability, all balls of radius $\delta\log_d(n)$ have tree excess at most 1. Indeed, due to statement 1 of Claim \ref{claim-cond}, with high probability, the $2\delta\log_d(n)$ neighborhood of $p$ has no cycles of length $\le2\delta\log_d(n)$. From now on, we work on this event. Run the exploration process from a vertex $v$, without avoiding the $2\delta\log_d(n)$ neighborhood of $p$. Because there are no cycles of length $\le2\delta\log_d(n)$ in the revealed part, the only way to close a cycle in $B(v,\delta\log(n))$ is with an unexplored edge. Therefore, the number of cycles that actually close is stochastically dominated by a Bin$(c'^4\cdot n^{2\delta}\cdot\log(n)^2,d/n)$, so
$$\PR(\text{tx}\left(B(v,\delta\log_d(n)))\ge2\ |\ E_p\right)\le c'^2d^2\cdot\log(n)^4\cdot n^{-2+4\delta}=o(n^{-1}).$$
Finally, we prove that conditional on $E_p$, with high probability, all $B(v,r_v)$ have tree excess at most 1. Due to statement 2 of Claim \ref{claim-cond}, it suffices to work on the event that for every $v$, $|B(v,r_v)|<n^{\delta/2}$. By statement 3 of Claim \ref{claim-cond}, similarly as before, the only cycles that can close inside $B(v,r_v)$ are the ones with unexplored edges. Indeed, since there is no pair of vertices $(v_i,v_j)$ such that $|i-j|\le\log(n)^2$ and both were discovered naturally, it is impossible for a cycle to close due to $p$ being present. So, the tree excess is stochastically dominated by a $\text{Bin}(n^\delta,d/n)$, therefore
$$\PR\left(\text{tx}(B(v,r_v)\ge2\ |\ E_p\right)\le (d\cdot n^{-1+\delta})^2=o(n^{-1}).$$
The result follows.
\end{proof}
\subsection{Counting Numbers of Walks}\label{4.4}
In this subsection, we perform the moment calculations necessary to prove three lemmas regarding numbers of walks in $G(n,d/n)$, and conclude the proof of Proposition \ref{ER-whp-prop}.
\begin{lem}\label{average-vertex}
For some $c_1>0$,
$$\dfrac{1}{n}\SL{x}N_x^{(L)}\ge c_1\cdot d^L,$$
with high probability.
\end{lem}
\begin{lem}\label{mid-range}
There exists some $c>0$ for which
$$N_v^{(\ell)}\le c\cdot d^{\ell-r_v}\cdot|\partial B(v,r_v)|+2\cdot|B(v,r_v)|\le c^2\cdot (d^\ell+n^{\delta/10})\le c^3\cdot d^\ell,$$
for all $v\in B$ and $r_v\le\ell\le K\log_d(n)$, with high probability.
\end{lem}
\begin{lem}\label{task5-average}
For the random graph $G(n,d/n)$, there exists $C>0$ such that the inequality
$$\dfrac{1}{n}\cdot\SL{\ell=L+1}^n\SL{x}\SL{y}N_{xy}^{(L)}\cdot S_y^{(\ell-L)}\cdot d^{-\ell}\cdot\dfrac{1}{\ell}\le C\cdot\log(n)$$
holds with high probability.
\end{lem}
Next, we make some observations that will be used in the proofs that follow.
\begin{obs}\label{key-obss}
\begin{enumerate}
\item 
Under $A_n$, for each vertex $v$ and each $1\le\ell\le\max(r_v,\delta\log_d(n))$, $N_x^{(\ell)}\le2\cdot|B(x,\ell)|$. Indeed, to a specific destination within $B(v,\ell)$, there can only exist two Non-Backtracking walks of length $\ell$ (because there is at most one cycle). The same inequality holds for the number of walks originating from $x$ that do not leave $B_\ell(v)$, for any length.
\item
Under $A_n$, if $n$ is large, for every $\ell\le K\log_d(n)$ we have that $N_x^{(\ell)}\le d^\ell\cdot(\log(n))^{2K/\delta}$. Indeed, this follows directly from concatenation of walks of length $\delta\log_d(n)$ and the fact that we are under $A_n$.
\end{enumerate}
\end{obs}
Before moving on to the proofs, we explain why those lemmas combined with the rest of the work in this section is enough to prove that condition 4 holds with high probability in $G(n,d/n)$. 
\begin{proof}[Proof of Proposition \ref{ER-whp-prop}]
As we explained in subsection \ref{4.1}, the first three properties in Theorem \ref{main-thm-determ} are satisfied by $G(n,d/n)$ with high probability. We turn to property 4.\\\\
We work on the event $A_n$. 
For $1\le\ell\le\delta\log_d(n)$,
$$S_v^{(\ell)}\le2\cdot|B(v,\ell)|\le C'\cdot d^\ell.$$
If $r_v=\infty$, this inequality is trivially satisfied for all $\ell\in\N$, as the connected component of $v$ contains at most $O(\log_d(n))$ vertices.\\\\
If $r_v<\delta\log_d(n)$, due to Lemma \ref{mid-range},
$S_v^{(\ell)}\le C'\cdot d^{\ell}$
for any $\ell\le K\log_d(n)$, so $$S_v^{(\ell)}\le C'\cdot d^\ell,\ \ \text{for all}\ \ 1\le\ell\le K\log_d(n).$$
If $r_v\ge\delta\log_d(n)$, for every $\delta\log_d(n)\le\ell\le r_v$,
$$S_v^{(\ell)}\le C'\cdot n^{\delta/10}= C'\cdot d^{\delta\log_d(n)/10}\le C'\cdot d^\ell,$$
while for $\ell\in[1,\delta\log_d(n)]\cup[r_v,K\log_d(n)]$ this inequality is true, as we explained above.\\\\
Next, let $\ell>K\log_d(n)$. Due to Lemmas \ref{mixing-NB}, \ref{average-vertex} and \ref{mid-range}, the following chain of inequalities is true with high probability:
\begin{align*}
N_{zy}^{(L)}&\le\dfrac{v(y)}{\lVert v\rVert_1}\cdot N_z^{(L)}+d^{(K+10)\log_d(n)/2}\\&\le\dfrac{v(y)}{\lVert v\rVert_1}\cdot c^3\cdot c_1^{-1}\cdot\dfrac{1}{n}\SL{x}N_x^{(L)}+d^{(K+10)\log_d(n)/2}\\&\le c^3\cdot c_1^{-1}\cdot\dfrac{1}{n}\SL{x}N_{xy}^{(L)}+(1+c^3\cdot c_1^{-1})\cdot d^{(K+10)\log_d(n)/2}.
\end{align*}
Thus,
\begin{align*}
S_v^{(\ell)}&\le\SL{y}N_{vy}^{(L)}\cdot S_y^{(\ell-L)}\\&\le C\cdot\left(\dfrac{1}{n}\SL{x}\SL{y}N_{xy}^{(L)}\cdot S_y^{(\ell-L)}+d^{(K+10)\log_d(n)/2}\cdot\SL{y}S_y^{(\ell-L)}\right).
\end{align*}
Therefore, due to Lemma \ref{small-error}, for each $\ell>K\log_d(n),$
$$\sup\limits_{v\in B}S_v^{(\ell)}\le1+\dfrac{C}{n}\cdot\SL{x}\SL{y}N_{xy}^{(L)}\cdot S_y^{(\ell-L)}.$$
Putting everything together and using Lemma \ref{task5-average},
\begin{align*}
&\SL{\ell=1}^nd^{-\ell}\cdot\dfrac{1}{\ell}\cdot\sup\limits_{v\in B}S_v^{(\ell)}\\\le\ & C'\cdot K\cdot\log_d(n)+\SL{\ell>K\log_d(n)}d^{-\ell}\cdot\dfrac{1}{\ell}\cdot\left(1+\dfrac{C}{n}\cdot\SL{x}\SL{y}N_{xy}^{(L)}\cdot S_y^{(\ell-L)}\right)
\\\le\ &C'\cdot K\cdot\log_d(n)+(1-d^{-1})^{-1}+C^2\log(n)\\\le\ &C''\cdot\log(n),
\end{align*}
which is exactly what we intended to prove.
\end{proof}
\begin{proof}[Proof of Lemma \ref{average-vertex}]
We perform a second moment calculation. Our intention is to prove that
\begin{align}\label{2-mom-rel}
\E\left[\left(\dfrac{1}{n}\SL{x}N_x^{(L)}\right)^2\cdot\mathbf{1}_{A_n}\right]\le(1+o(1))\cdot\E\left[\left(\dfrac{1}{n}\SL{x}N_x^{(L)}\right)\cdot\mathbf{1}_{A_n}\right]^2,
\end{align}
and then use Chebyshev's Inequality. Considering only walks that are Self-Avoiding and avoid each other, together with Lemma \ref{cond-properties}, implies that
\begin{align}\label{average-1-mom}
\E\left[\left(\dfrac{1}{n}\SL{x}N_x^{(L)}\right)\cdot\mathbf{1}_{A_n}\right]\ge(1-o(1))\cdot d^L.
\end{align}
We study the Left-Hand Side of (\ref{2-mom-rel}). A pair $(w_1,w_2)$ of walks on $K_n$ is called \textit{respectful} if $\{w_1\in G\}\cap\{w_2\in G\}\cap A_n\neq\varnothing$, i.e. if there is a realization of $G$ in which $w_1,w_2$ are present and $A_n$ is satisfied. Denote by $\mathcal{R}^{(2)}$ the set of respectful pairs $(w_1,w_2)$. Of course,
\begin{align*}
\E\left[\left(\dfrac{1}{n}\SL{x}N_x^{(L)}\right)^2\cdot\mathbf{1}_{A_n}\right]&=\dfrac{1}{n^2}\cdot\SL{(w_1,w_2)}\PR\left(\{w_1\in G\}\cap\{w_2\in G\}\cap A_n\right)\\&\le\dfrac{1}{n^2}\cdot\SL{(w_1,w_2)\in\mathcal{R}^{(2)}}\PR(\{w_1\in G\}\cap\{w_2\in G\}),
\end{align*}
where both sums run over all possible pairs of Non-Backtracking walks $(w_1,w_2)$ of length $L$ in the complete graph $K_n$.
Let $(w_1,w_2)$ be one such pair and $w_1\cup w_2$ be the set of edges covered by either of those walks. Then, due to the independence in $G(n,d/n)$,
$$\PR(\{w_1\in G\}\cap\{w_2\in G\})=\left(\dfrac{d}{n}\right)^{|w_1\cup w_2|}.$$
In order to analyze $|w_1\cup w_2|$, we run $w_1$ and $w_2$ sequentially, first going through $w_1$ and then through $w_2$. Throughout this process, we keep track of the edges that have been revealed, the set of which is called the \textit{past} of the process. At each time, the process is in one of the following two modes: It either explores a new edge, if the edge that is being crossed has never been crossed in the past of the process, or it visits the past, if the edge being crossed has already been revealed. Therefore, when we run $w_1$, we are initially exploring new edges. The same happens with $w_2$ if its starting point was not visited by $w_1$, otherwise we start by visiting the past, with this part of the walk lasting 0 steps if necessary. Assume the following about the pair of walks:
\begin{itemize}
\item
Let $a=1$ or $0$, if the starting point of $w_2$ was visited by $w_1$ or not, respectively.
\item 
Each walk $w_i$ makes $m_i$ mode shifts between exploring new vertices and visiting the past, with the lengths of these trips being $\gamma_1^{(i)},...,\gamma_{m_i}^{(i)}$ (for $i=1,2$). In other words: For $w_1$, for even $k$, at times $t\in\left[1+\SL{j=1}^k\gamma_j^{(1)},\SL{j=1}^{k+1}\gamma_j^{(1)}\right]$, the walk explores new edges, while for odd $k$, between times $t\in\left[1+\SL{j=1}^k\gamma_j^{(1)},\SL{j=1}^{k+1}\gamma_j^{(1)}\right]$ the walk visits the past. Also, for $w_2$, for $k\equiv a (\bmod2)$, between times $t\in\left[1+\SL{j=1}^k\gamma_j^{(2)},\SL{j=1}^{k+1}\gamma_j^{(2)}\right]$ the walk explores new edges and for $k\neq a(\bmod2)$, between times $t\in\left[1+\SL{j=1}^k\gamma_j^{(2)},\SL{j=1}^{k+1}\gamma_j^{(2)}\right]$ the walk visits the past. The lengths $\gamma_j^{(i)}$ satisfy the equations
\begin{align}\label{gamma-eq-2}
\gamma_1^{(i)}+\cdots+\gamma_{m_i}^{(i)}=L.
\end{align}
\item 
Let $M$ be the number of trips to the past, which happened after an exploration of new vertices within the same walk $w_i$. From this description, we can see that
\begin{align}\label{M-eq-2}
M=\left\lfloor\dfrac{m_1}{2}\right\rfloor+\left\lfloor\dfrac{m_2-a}{2}\right\rfloor.
\end{align}
\end{itemize}
If the pair $(w_1,w_2)$ satisfies these assumptions, since $|w_1\cup w_2|$ is exactly the time spent exploring new edges,
\begin{align}\label{prob-both-in}
\PR(\{w_1\in G\}\cap\{w_2\in G\})=\left(\dfrac{d}{n}\right)^{|w_1\cup w_2|}=\left(\dfrac{d}{n}\right)^{\SL{j\equiv1(\bmod2)}\gamma_j^{(1)}+\SL{j\neq a(\bmod2)}\gamma_j^{(2)}}.
\end{align}
Next, we bound the number of respectful pairs of walks in $K_n$ that satisfy these assumptions.\\\\ 
If $(w_1,w_2)\in\mathcal{R}^{(2)}$, due to Observation \ref{key-obss}, during each trip to the past of length $\gamma$, the number ways to carry out this trip is at most $d^\gamma\cdot(\log(n))^{2K/\delta}$.\\\\ 
Consider a trip discovering new edges of length $\gamma$. If there is no shift of the mode to the past afterward, the trip can be carried out in at most $n^\gamma$ ways. However, if there is a shift of the mode to the past afterward, it can only be performed in $n^{\gamma-1}\cdot2K\log_d(n)$ ways, as there are at most $2K\log(n)$ choices for the starting location of the trip to the past that follows.\\\\ 
If $a=0$, the number of ways to choose the pair of starting points is at most $n^2$, whereas if $a=1$, the number of ways to choose the pair of starting points is at most $Kn\log_d(n)$, as the second starting points can only be among the ones visited by $w_1$.
So, the number of pairs of paths in $K_n$ satisfying these conditions is at most
\begin{align*}
n^2\cdot d^{\SL{j\equiv0(\bmod2)}\gamma_j^{(1)}+\SL{j\equiv a(\bmod2)}\gamma_j^{(2)}}\cdot(\log(n)^{2K/\delta})^{M+a}&\cdot n^{\SL{j\equiv1(\bmod2)}\gamma_j^{(1)}+\SL{j\neq a(\bmod2)}\gamma_j^{(2)}}\\&\cdot\left(\dfrac{c_0\cdot\log(n)}{n}\right)^{M+a}.
\end{align*}
Due to (\ref{prob-both-in}) and (\ref{gamma-eq-2}), the expected number of paths in $G$ satisfying these conditions is at most
$$n^2\cdot d^{2L}\cdot\left(\dfrac{c_0\log(n)^{1+2K/\delta}}{n}\right)^{M+a}\le n^2\cdot d^{2L}\cdot\left(\dfrac{c_0\log(n)^{1+2K/\delta}}{n}\right)^{M+a}.$$
The number of solutions to equation (\ref{gamma-eq-2}) is at most $(K\log_d(n))^{m_i-1}$ and
$$(m_1-1)+(m_2-1)=(m_1-1)+(m_2-1-a)+a\le2M+a.$$
Also, the number $n_{M,a}$ of solutions to (\ref{M-eq-2}) is at most
$$
\begin{cases}
1,\ \text{if}\ M=a=0\\
2,\ \text{if}\ M=0\ \text{and}\ a=1\\
4(M+1),\ \text{if}\ M\ge1.
\end{cases}
$$
Putting all of these together,
\begin{align*}
&\E\left[\left(\dfrac{1}{n}\SL{x}N_x^{(L)}\right)^2\cdot\mathbf{1}_{A_n}\right]\\\le\ &
\dfrac{1}{n^2}\cdot\SL{(w_1,w_2)\in\mathcal{R}^{(2)}}\PR(\{w_1\in G\}\cap\{w_2\in G\})\\ \le\ &d^{2L}\cdot\SL{a\in\{0,1\}}\SL{M\ge0}\left(\dfrac{c_0\log(n)^{3+2K/\delta}}{n}\right)^{M+a}\cdot n_{M,a}
\\ \le\ &d^{2L}\cdot\SL{M\ge0}(M+1)\cdot\left(1+\dfrac{2c_0\log(n)^{3+2K/\delta}}{n}\right)\cdot\left(\dfrac{c_0\log(n)^{3+2K/\delta}}{n}\right)^{M}
\\=\ &(1+o(1))\cdot d^{2L}.
\end{align*}
So, if
$$Y_n:=\dfrac{1}{n}\SL{x}N_x^{(L)},$$
due to Chebyshev's inequality,
\begin{align*}
\PR\left(Y_n<\dfrac{1}{2}\cdot \E(Y_n)\right)\le\PR(A_n^c)+\dfrac{4\cdot\E\left[\left(Y_n-\E(Y_n)\right)^2\cdot\mathbf{1}_{A_n}\right]}{\E(Y_n\cdot\mathbf{1}_{A_n})^2}=o(1),
\end{align*}
as $n\to\infty$, which is what we wanted.
\end{proof}
\begin{proof}[Proof of Lemma \ref{mid-range}]
In the proof of Lemma \ref{tx-1}, we proved that with high probability, for all vertices $v\in V$, $|B(v,r_v)|\le2dn^{\delta/10}$. Also, for a vertex $v\in B$, due to Remark \ref{balls-slowly},
$$|B(v,\delta\log_d(n)/10-\log_d(C'))|\le C\cdot(1-d^{-1})^{-1}\cdot \dfrac{n^{\delta/10}}{C'},$$
hence if $C'$ is large enough, $r_v\ge\delta\log_d(n)/10-\log(C')$ and $d^{r_v}\ge n^{\delta/10}/C'$. All of these inequalities imply the second and third inequality of the lemma.\\\\
We now aim to prove the first inequality. We condition on $B_{r_v}(v)$. Consider the graph $K_v$, which is the same as $K_n$, except for the edges that are not present due to conditioning on $B_{r_v}(v)$. In other words, in $K_v$, the neighbors of vertices in $B_{r_v}(v)\setminus\partial B_{r_v}(v)$ are only the ones imposed by $B_{r_v}(v)$, whereas all the other edges are present. For any $\ell$, due to Observation \ref{key-obss}, under $A_n$, the number of walks of length $\ell$ not exiting $B_{r_v}(v)$ is at most $2\cdot|B_{r_v}(v)|$. Also, for any $w\in\partial B_{r_v}(v)$ and $1\le\ell'\le\ell-r_v$, we denote by $\tilde{N}_w^{(\ell')}$ to be the number of Non-Backtracking walks originating from $w$ that at their first step leave $B_{r_v}(v)$. Of course, for any $\ell$, conditioned on $B_{r_v}(v)$,
$$N_v^\ell\le2\cdot\SL{\ell'=1}^{\ell-r_v}\SL{w\in\partial B_{r_v}(v)}\tilde{N}_w^{(\ell')}+2\cdot|B_{r_v(v)}|.$$
Also,
\begin{align*}
&\PR\left(\left\{N_v^{(\ell)}>2c\cdot d^{\ell-r_v}\cdot|\partial B_{r_v}(v)|+2\cdot|B_{r_v}(v)|\right\}\cap A_n\right)\\\le\ &
\E\left[\PR\left(\left\{N_v^{(\ell)}>2c\cdot d^{\ell-r_v}\cdot|\partial B_{r_v}(v)|+2\cdot|B_{r_v}(v)|\right\}\cap A_n\right)|B_{r_v}(v)\right]\\\le\ &
\E\left[\PR\left(\left\{\SL{\ell'=1}^{\ell-r_v}\SL{w\in\partial B_{r_v}(v)}\tilde{N}_w^{(\ell')}>c\cdot d^{\ell-r_v}\cdot|\partial B_{r_v}(v)|\right\}\cap A_n\right)|B_{r_v}(v)\right].
\end{align*}
Let $s=\lceil\log(n)\rceil$. We want to prove that
\begin{align}\label{s-mom-rel}
\E\left[\left(\SL{\ell'=1}^{\ell-r_v}\SL{w\in\partial B_{r_v}(v)}\tilde{N}_w^{\ell'}\right)^s\mathbf{1}_{A_n}|B_{r_v}(v)\right]\le(1+o(1))\cdot\E\left[\SL{\ell'=1}^{\ell-r_v}\SL{w\in\partial B_{r_v}(v)}\tilde{N}_w^{\ell'}|B_{r_v}(v)\right]^s.
\end{align}
At first, considering only $s$-tuples of walks that are Self-Avoiding and non-intersecting with one another, we can see that
\begin{align*}
&\E\left[\SL{\ell'=1}^{\ell-r_v}\SL{w\in\partial B_{r_v}(v)}\tilde{N}_w^{\ell'}|B_{r_v}(v)\right]\ge\left(1-O\left(\dfrac{\log^2(n)}{n}\right)\right)\cdot|\partial B_{r_v}(v)|\cdot\SL{\ell'=1}^{\ell-r_v}d^{\ell'}\\\Rightarrow\ &
\E\left[\SL{\ell'=1}^{\ell-r_v}\SL{w\in\partial B_{r_v}(v)}\tilde{N}_w^{\ell'}|B_{r_v}(v)\right]^s\ge(1-o(1))\cdot|\partial B_{r_v}(v)|^s\cdot\left(\SL{\ell'=1}^{\ell-r_v}d^{\ell'}\right)^s.
\end{align*}
We now proceed to estimate the Left-Hand Side of (\ref{s-mom-rel}), by considering all possible $s$-tuples of walks starting from $\partial B_{r_v}(v)$ and leaving $B_{r_v}(v)$ at the first step. Similarly as in the proof of Lemma \ref{average-vertex}, an $s$-tuple of walks $(w_1,w_2,...,w_s)$ on $K_v$ will be called \textit{respectful}, if $\{w_1\in G\}\cap\cdots\cap\{w_s\in G\}\cap A_n\neq\varnothing$, i.e. if there is a realization of $G$ in which $w_1,...,w_s$ are present and $A_n$ is satisfied. Call $\mathcal{R}^{(s)}$ the set of $s$-tuples of respectful walks. Then, similarly as in the proof of Lemma \ref{average-vertex},
$$\E\left[\left(\SL{\ell'=1}^{\ell-r_v}\SL{w\in\partial B_{r_v}(v)}\tilde{N}_w^{\ell'}\right)^s\mathbf{1}_{A_n}|B_{r_v}(v)\right]\le\SL{(w_1,...,w_s)\in\mathcal{R}^{(s)}}\PR(w_i\in G,\ \forall i\ |B_{r_v}(v)),$$
where the sum is over all respectful $s$-tuples of walks $(w_1,w_2,...,w_s)$ of lengths $\ell'\in[1,\ell-r_v]$.
We run the walks sequentially as before, first going through $w_1$, then $w_2$, etc., until $w_s$. Throughout this process, we keep track of the edges that have been revealed, which, together with the edges of $B_{r_v}(v)$ is called the past of the process. The starting point of a walk $w_i$ (or, sometimes, $i$ itself, for simplicity) will be called \textit{original} if it has not been visited by any of the walks $w_1,...,w_{i-1}$ and \textit{visited} otherwise. Similarly as in the proof of Lemma \ref{average-vertex}, if the starting point of $w_i$ is visited, we will assume that the walk starts from visiting the past, with this trip lasting 0 steps, if necessary. Suppose an $s$-tuple of walks $(w_1,w_2,...,w_s)$ satisfies the following:
\begin{itemize}
\item 
The walks $w_1,...,w_s$ have lengths $\ell_1,...,\ell_s$, respectively.
\item 
Exactly $a$ of the starting points of the walks are visited.
\item 
Each walk $w_i$ makes $m_i$ shifts of mode between exploring new edges and visiting the past, with the lengths of these trips being $\gamma_1^{(i)},\gamma_2^{(i)},...,\gamma_{m_i}^{(i)}$, for $1\le i\le s$. In other words: If the starting point of $w_i$ is original, for even $k$, at times $t\in\left[1+\SL{j=1}^k\gamma_j^{(i)},\SL{j=1}^{k+1}\gamma_j^{(i)}\right]$ $w_i$ explores new edges and for odd $k$, at times $t\in\left[1+\SL{j=1}^k\gamma_j^{(i)},\SL{j=1}^{k+1}\gamma_j^{(i)}\right]$ the walk $w_i$ visits the past. If the starting point of $w_i$ is visited, for odd $k$, at times $t\in\left[1+\SL{j=1}^k\gamma_j^{(i)},\SL{j=1}^{k+1}\gamma_j^{(i)}\right]$ $w_i$ explores new edges and for even $k$, at times $t\in\left[1+\SL{j=1}^k\gamma_j^{(i)},\SL{j=1}^{k+1}\gamma_j^{(i)}\right]$ the walk $w_i$ visits the past. The lengths $\gamma_j^{(i)}$ satisfy the equations 
\begin{align}\label{gamma-eq}
\gamma_1^{(i)}+\gamma_2^{(i)}+...+\gamma_{m_i}^{(i)}=\ell_i.
\end{align}
\item
Let $M$ be the number of trips to the past, which happened after an exploration of new vertices within the same walk $w_i$ (i.e. the trips to the past started at touched or unoriginal vertices are the only ones that do not count towards $M$). From this description, we deduce that
\begin{align}\label{M-eq}
M=\SL{i:\ \text{original}}\left\lfloor\dfrac{m_i}{2}\right\rfloor+\SL{i:\ \text{visited}}\left\lfloor\dfrac{m_i-1}{2}\right\rfloor.
\end{align}
\end{itemize}
If the $s$-tuple $(w_1,\dots,w_s)$ satisfies these assumptions, due to the independence in $G(n,d/n)$,
\begin{equation}\label{prob-all-in}
\PR(w_i\in G,\ \forall i\ |B_{r_v}(v))=\left(\dfrac{d}{n}\right)^{\SL{i: \text{original}}\SL{j\equiv1(\bmod2)}\gamma_j^{(i)}+\SL{i:\text{visited}}\SL{j\equiv0(\bmod2)}\gamma_j^{(i)}}.
\end{equation}
We now calculate the number of respectful $s$-tuples of walks in $K_v$ that satisfy these assumptions. \\\\
Because $(w_1,\dots,w_s)\in\mathcal{R}^{(s)}$, due to Observation \ref{key-obss}, during each trip to the past of length $\gamma$, the number of ways to carry out this trip is at most $d^{\gamma}\cdot\log(n)^{2K/\delta}$.\\\\
Consider a trip discovering new edges of length $\gamma$. If there is no shift of mode to the past afterwards, the trip can be carried out in at most $n^\gamma$ ways. However, if a trip to the past follows, there are at most $n^{\gamma-1}\cdot(cn^{\delta/10}+c\log(n)^2)\le2c\cdot n^{\gamma-1+\delta/10}$ ways to carry out the trip, as there are at most $cn^{\delta/10}+c\log(n)^2$ possibilities for the starting point of the trip to the past that follows.\\\\
There are $\binom{s}{a}$ ways of choosing the set of $i$ which will be original. When $i$ is original, there are at most $|\partial B_{r_v}(v)|$ ways to choose the starting point of $w_i$, whereas when $i$ is visited, there are at most $c\log(n)^2$ ways to choose the starting point of $w_i$. Therefore, there are at most
$$\dbinom{s}{a}\cdot|\partial B_{r_v}(v)|^{s-a}\cdot(c\log(n)^2)^a\le|\partial B_{r_v}(v)|^s\cdot\left(\dfrac{c_0\log(n)^3}{n^{\delta/10}}\right)^a$$
ways to choose the starting points of the $w_i$'s. Working as in the proof of Lemma \ref{average-vertex}, i.e. combining (\ref{prob-all-in}) with the counting bounds we mentioned, we find that the expected number of respectful $s$-tuples of walks in the graph satisfying the above conditions is at most
\begin{align*}
&|\partial B_{r_v}(v)|^s\cdot\left(\dfrac{c_0\log(n)^3}{n^{\delta/10}}\right)^a\cdot d^{\ell_1+\cdots+\ell_s}\cdot(\log(n)^{2K/\delta})^{M+a}\cdot\left(\dfrac{c_0\log(n)}{n^{1-\delta/10}}\right)^M\\\le\ & |\partial B_{r_v}(v)|^s\cdot d^{\ell_1+\ell_2+\cdots+\ell_s}\cdot\left(\dfrac{c_0\log(n)^{3+2K/\delta}}{n^{\delta/10}}\right)^a\cdot\left(\dfrac{c_0\cdot\log(n)^{1+2K/\delta}}{n^{1-\delta/10}}\right)^M.
\end{align*}
Observe that the number of solutions to equation (\ref{gamma-eq}) is at most $\ell_i^{m_i-1}\le(K\log_d(n))^{m_i-1}$ and that
\begin{align*}
\SL{i}(m_i-1)=\SL{i:\text{original}}(m_i-1)+\SL{i:\text{visited}}(m_i-2)+\SL{i:\text{visited}}1\le2M+a.
\end{align*}
We also need to study the number of solutions to equation (\ref{M-eq}). We prove that for any $M$ and $a$, the number of solutions is at most $(K\log_d(n))^a\cdot(8s)^M$.\\\\ 
For $M>s$, there are $\binom{s+M-1}{s-1}$ ways to choose the values of the integer parts and at most $2^s$ ways to realize these values. Therefore, since $M>s$, the number of solutions in this case is 
$$\le\dbinom{s+M-1}{s-1}\cdot2^s\le2^{s+M}\cdot2^s\le8^M\le(K\log_d(n))^a\cdot(8s)^M.$$
Let $M\le s$. Since $m_i\le K\log_d(n)$ for each visited $i$, there are at most $(K\log_d(n))^a$ choices for the values of the $a$-tuple of $m_i$'s, when $i$ is visited. For each one of these, the value (call it $q$) of $\SL{i:\text{original}}\lfloor\frac{m_i}{2}\rfloor$ is uniquely determined and $\le M$. Observe that at most $M$ of the terms in the above sum are non-zero. Since the number of terms is at most $s$, so there are at most $s^M$ ways of choosing $M$ of them that are candidates to be non-zero. Observe that for the terms that are chosen to be 0, the respective $m_i$'s are determined to be equal to 1. As far as the terms that are chosen to have the possibility of being non-zero, the number of ways to choose the $m_i$'s after we have chosen the values of the integer parts is at most $2^M$. The number of ways to choose the values of the integer parts is $\le\binom{q+M}{M}\le\binom{2M}{M}\le4^M$. In total, there are at most
$$(K\log_d(n))^a\cdot s^M\cdot2^M\cdot4^M=(K\log_d(n))^a\cdot(8s)^M$$
solutions to Equation (\ref{M-eq}) in this case as well.\\\\
Putting everything together,
\begin{align*}
&\E\left[\left(\SL{\ell'=1}^{\ell-r_v}\SL{w\in\partial B_{r_v}(v)}\tilde{N}_w^{\ell'}\right)^s\mathbf{1}_{A_n}|B_{r_v}(v)\right]\\\le\ &\SL{(w_1,...,w_s)\in\mathcal{R}^{(s)}}\PR(w_i\in G,\ \forall i\ |B_{r_v}(v))\\\le\
&|\partial B_{r_v}(v)|^s\cdot\SL{\ell_1,\dots,\ell_s}d^{\ell_1+\cdots+\ell_s}\SL{a\ge0}\SL{M\ge0}\left(\dfrac{c_0\log(n)^{4+2K/\delta}}{n^{\delta/10}}\right)^a\cdot\left(\dfrac{c_0\cdot\log(n)^{2+2K/\delta}}{n^{1-\delta/10}}\right)^M
\\=\ & (1+o(1))\cdot|\partial B_{r_v(v)}|^s\cdot\left(\SL{\ell'=1}^{\ell-r_v}d^{\ell'}\right)^s.
\end{align*}
We have proven (\ref{s-mom-rel}). For a vertex $v$, set
$$X_v:=\SL{\ell'=1}^{\ell-r_v}\SL{w\in\partial B_{r_v}(v)}\tilde{N}_w^{\ell'}.$$ 
Using Markov's inequality, we can deduce that for each $v,\ell$, and for every $c>0$,
\begin{align*}
&\PR\left(\left\{N_v^{(\ell)}>2c\cdot d^{\ell-r_v}\cdot|\partial B_{r_v}(v)|+2\cdot|B_{r_v}(v)|\right\}\cap A_n\right)\\\le\ &\E\left[\PR\left(X_v\cdot\mathbf{1}_{A_n}>c\cdot d^{\ell-r_v}\cdot|\partial B_{r_v}(v)|\right)|B_{r_v}(v)\right]\\\le\
&\E\left[\PR\left(X_v\cdot\mathbf{1}_{A_n}>\dfrac{c}{2\cdot(1-d^{-1})^{-1}}\cdot\E\left[X_v|B_{r_v}(v)\right]\right)|B_{r_v}(v)\right]\\\le\
&\E\left[\dfrac{\E\left[X_v^s\cdot\mathbf{1}_{A_n}|B_{r_v}(v)\right]}{\E\left[X_v|B_{r_v}(v)\right]^s}\ |B_{r_v}(v)\right]\cdot\left(\dfrac{2\cdot(1-d^{-1})^{-1}}{c}\right)^s
\\=\ &(1+o(1))\cdot\left(\dfrac{2\cdot(1-d^{-1})^{-1}}{c}\right)^s.
\end{align*}
We let $c>0$ be large enough, so that this bound is $o(n^{-2})$. Then, taking a union bound over values of $\ell$ and vertices $v$ and keeping in mind that $A_n$ is a high probability event finishes the proof of Lemma \ref{mid-range}.
\end{proof}
\begin{proof}[Proof of Lemma \ref{task5-average}]
Observe that for any $x\in V$ and $\ell\ge L$, the quantity
$$C_x^{(\ell)}:=\SL{y}N_{xy}^{(L)}\cdot S_y^{(\ell-L)}$$
counts the number of walks originating from $x$ that are for the first $L$ steps Non-Backtracking and for the next $\ell-L$ steps Self-Avoiding. From now on, when we talk about a walk, we assume that these are the restrictions on it. At first, we reduce the sum into counting of walks of length at most $\sqrt{n}/e^{\sqrt{\log(n)}}$. Indeed, at first observe that if $\ell\ge\sqrt{n}\cdot e^{\sqrt{\log(n)}}$, during the Non-Backtracking part of the walk, at most $K\log_d(n)$ edges are created. The number of Self-Avoiding paths during the second part of the walk is at most
$$(n-1)(n-2)\cdots(n-\ell+K\log_d(n))\le n^\ell\cdot\exp(-\ell^2/3n)\le n^\ell\cdot\exp\left(-e^{2\sqrt{\log(n)}}/3\right).$$
Therefore, since for each path all but at most $K\log_d(n)$ of the edges are unexplored,
\begin{align*}
\E(C_x^{(\ell)})&\le n^{K\log_d(n)}\cdot\left(\dfrac{d}{n}\right)^{\ell-K\log_d(n)}\cdot n^\ell\cdot\exp\left(-e^{2\sqrt{\log(n)}}/3\right)\\&\le d^\ell\cdot\exp\left(-e^{2\sqrt{\log(n)}}/3+cK\log(n)^2\right)\\&=o(1)\cdot d^\ell.
\end{align*}
We conclude that
\begin{align*}
&\E\left[\dfrac{1}{n}\SL{x}\SL{\ell\ge\sqrt{n}\cdot e^{\sqrt{\log(n)}}}d^{-\ell}\cdot\dfrac{1}{\ell}\cdot C_x^{(\ell)}\right]=o(\log(n))\\\Rightarrow\ &
\PR\left(\dfrac{1}{n}\SL{x}\SL{\ell\ge\sqrt{n}\cdot e^{\sqrt{\log(n)}}}d^{-\ell}\cdot\dfrac{1}{\ell}\cdot C_x^{(\ell)}>\log(n)\right)=o(1),
\end{align*}
which is what we wanted for this part of the sum.\\\\ 
Next, we treat the part of the sum for $\sqrt{n}/e^{\sqrt{\log(n)}}\le\ell\le\sqrt{n}\cdot e^{\sqrt{\log(n)}}.$ We bound the quantity $\E(C_x^{(\ell)}\cdot\mathbf{1}_{A_n})$, using similar techniques as before. Let $w$ be a walk of length $\ell$ that is Non-Backtracking for the first $K\log_d(n)$ steps and Self-Avoiding for the next $\ell-L$ steps. As before, let $m$ be the number of total shifts of mode, from exploring new edges to returning to the past, with the lengths of these trips being $\gamma_1,...,\gamma_m$ and satisfying the equation
\begin{align}\label{gamma-eq-3}
\gamma_1+\cdots+\gamma_m=\ell.
\end{align}
Observe that the number of solutions to (\ref{gamma-eq-3}) is at most $\ell^{\lceil\frac{m}{2}\rceil-1}\cdot(K\log_d(n))^{\lfloor\frac{m}{2}\rfloor}$. Indeed, the $\gamma_i$ that correspond to trips to the past (i.e. the ones with even $i$) cannot be longer than $K\log_d(n)$, as after $K\log_d(n)$ steps the walk is Self-Avoiding. This means that the only possible visits to the past happen towards the Non-Backtracking part of the walk, so each visit cannot be longer than the size of this part. The variables $\gamma_i$ with odd $i$ are $\lceil\frac{m}{2}\rceil$ in total and their sum is a fixed number at most $\ell$, hence the bound we claimed. Therefore:
\begin{align*}
\E(C_x^{(\ell)}\cdot\mathbf{1}_{A_n})&\le d^\ell\cdot\SL{m}\left(\dfrac{c_0\log(n)}{n}\right)^{\lfloor \frac{m}{2}\rfloor}\cdot\left(\log(n)^{2K/\delta}\right)^{\lfloor\frac{m}{2}\rfloor}\cdot\ell^{\lfloor\frac{m}{2}\rfloor}\cdot(K\log(n))^{\lfloor\frac{m}{2}\rfloor}\\&=(1+o(1))\cdot d^\ell.
\end{align*}
This calculation combined with the fact that
$$\SL{\ell=\sqrt{n}/e^{\sqrt{\log(n)}}}^{\sqrt{n}\cdot e^{\sqrt{\log(n)}}}\dfrac{1}{\ell}=(2+o(1))\sqrt{\log(n)}=o(\log(n)),$$
implies
\begin{align*}
&\E\left[\dfrac{1}{n}\SL{x}\SL{\ell=\sqrt{n}/e^{\sqrt{\log(n)}}}^{\sqrt{n}\cdot e^{\sqrt{\log(n)}}}d^{-\ell}\cdot\dfrac{1}{\ell}\cdot C_x^{(\ell)}\right]=o(\log(n))\\\Rightarrow\ &
\PR\left(\dfrac{1}{n}\SL{x}\SL{\ell=\sqrt{n}/e^{\sqrt{\log(n)}}}^{\sqrt{n}\cdot e^{\sqrt{\log(n)}}}d^{-\ell}\cdot\dfrac{1}{\ell}\cdot C_x^{(\ell)}>\log(n)\right)=o(1).
\end{align*}
We are left to deal with the case $\ell\le\sqrt{n}/e^{\sqrt{\log(n)}}$. For that, we again perform a second moment calculation. At first, keeping in mind Lemma \ref{cond-properties} and summing only over Self-Avoiding walks $w$ of length $\ell$ we find that 
\begin{align*}
\E(C_x^{(\ell)}\cdot\mathbf{1}_{A_n})&\ge\SL{w}\PR(\{w\in G\}\cap A_n)
\\&=\SL{w}\PR(A_n|w\in G)\cdot\PR(w\in G)
\\&=(1-o(1))\cdot\SL{w}\PR(w\in G)
\\&=(1-o(1))\cdot(n-1)(n-2)\cdots(n-\ell)\cdot\left(\dfrac{d}{n}\right)^\ell
\\&=
(1-o(1))\cdot d^\ell,
\end{align*}
since $\ell\le\sqrt{n}/e^{\sqrt{\log(n)}}$. We can now see that
\begin{align}\label{1-mom-avg}
\E\left[\left(\dfrac{1}{n}\SL{x}\SL{\ell=1}^{\sqrt{n}/e^{\sqrt{\log(n)}}}d^{-\ell}\cdot C_x^{(\ell)}\cdot\dfrac{1}{\ell}\right)\cdot\mathbf{1}_{A_n}\right]\ge\left(\dfrac{1}{2}-o(1)\right)\cdot\log(n).
\end{align}
Next, we bound the quantity
\begin{align*}
&\E\left[\left(\dfrac{1}{n}\SL{x}\SL{\ell=1}^{\sqrt{n}/e^{\sqrt{\log(n)}}}d^{-\ell}\cdot C_x^{(\ell)}\cdot\dfrac{1}{\ell}\right)^2\cdot\mathbf{1}_{A_n}\right]
\\=\ &\SL{1\le\ell_1,\ell_2\le\sqrt{n}/e^{\sqrt{\log(n)}}}d^{-\ell_1+\ell_2}\cdot\dfrac{1}{\ell_1\ell_2}\cdot\E\left[\dfrac{1}{n^2}\left(\SL{x_1,x_2}C_{x_1}^{(\ell_1)}\cdot C_{x_2}^{(\ell_2)}\right)\cdot\mathbf{1}_{A_n}\right].
\end{align*}
More specifically, we prove that for any $\ell_1,\ell_2$,
$$\E\left[\dfrac{1}{n^2}\left(\SL{x_1,x_2}C_{x_1}^{(\ell_1)}\cdot C_{x_2}^{(\ell_2)}\right)\cdot\mathbf{1}_{A_n}\right]\le(1+o(1))\cdot d^{\ell_1+\ell_2},$$
which will imply that
\begin{equation}\label{sec-mom-avg}
\E\left[\left(\dfrac{1}{n}\SL{x}\SL{\ell=1}^{\sqrt{n}/e^{\sqrt{\log(n)}}}d^{-\ell}\cdot C_x^{(\ell)}\cdot\dfrac{1}{\ell}\right)^2\cdot\mathbf{1}_{A_n}\right]=\left(\dfrac{1}{4}+o(1)\right)\cdot\log(n)^2.
\end{equation}
Just as we did in the previous calculations, we write
$$\E\left[\dfrac{1}{n^2}\left(\SL{x_1,x_2}C_{x_1}^{(\ell_1)}\cdot C_{x_2}^{(\ell_2)}\right)\cdot\mathbf{1}_{A_n}\right]\le\dfrac{1}{n^2}\SL{(w_1,w_2)\in\mathcal{R}^{(2)}}\PR(\{w_1\in G\}\cap\{w_2\in G\}),$$
where the sum is over all respectful pairs of walks $(w_1,w_2)$ of length $\ell_1,\ell_2$ respectively, that are initially Non-Backtracking and afterwards Self-Avoiding. Consider such a pair $(w_1,w_2)$. We run the walks sequentially as in the proofs of Lemmas \ref{average-vertex} and \ref{mid-range}. For the walks, we again say that at any point in time, they can be in one of two modes, either visiting the past or exploring new edges. However, as far as $w_2$ is concerned, we will need to be a bit more careful in our analysis, so we recognize five different modes for this walk:
\begin{enumerate}
\item
We say that walk $w_2$ is in mode 1, when it is performing the Non-Backtracking part and is exploring new edges.
\item 
We say that $w_2$ is in mode 2, when it is performing the Non-Backtracking part and is visiting the past.
\item 
We say that $w_2$ is in mode 3, when it is performing the Self-Avoiding part and is exploring new edges.
\item 
We say that $w_2$ is in mode 4, when it is performing the Self-Avoiding part and is visiting edges previously visited by the NB parts of either of the two walks.
\item 
We say that $w_2$ is in mode 5, when it is performing the Self-Avoiding part and is visiting edges previously visited only by the Self-Avoiding part of $w_1$.
\end{enumerate}
To be consistent with previous terminology, we refer to modes 2, 4 and 5 as modes in which the walk is visiting the past, whereas modes 1 and 3 are modes in which new edges are being explored. Assume the following about the pair of walks $(w_1,w_2)$:
\begin{itemize}
\item 
Set $a=1$ or $0$, if the starting point of $w_2$ was visited by $w_1$ or not, respectively.
\item
Each walk $w_i$ makes $m_i$ shifts of mode between exploring new edges and visiting the past, with the lengths of these trips being $\gamma_1^{(i)},\gamma_2^{(i)},...,\gamma_{m_i}^{(i)}$, just as in the proof of Lemma \ref{average-vertex}. These lengths satisfy the equations
\begin{align}\label{gamma-eq-4}
\gamma_1^{(i)}+\cdots+\gamma_{m_i}^{(i)}=\ell_i.
\end{align}
\item 
The walk $w_2$ is in mode 5 for a total of $\gamma^{(5)}$ edges, and visits the past for $\gamma_p$ edges. Also, set $\alpha=\gamma_p-\gamma^{(5)}$. Then, we have the equations
\begin{align}\label{gamma-eq-5}
\SL{j\equiv a(\bmod2)}\gamma_j^{(2)}=\gamma_p\ \ \ \text{and}\ \ \SL{j\neq a(\bmod2)}\gamma_j^{(2)}=\ell_2-\gamma_p.
\end{align}
Moreover, we will use the fact that $\alpha\le3K\log_d(n)$, which is true since the Non-Backtracking part of $w_2$ has length $K\log_d(n)$ and in the Self-Avoiding part of $w_2$, the edges that can be used while in mode 4 are at most $2K\log_d(n)$.
\item 
During the whole process, a total of $M$ trips to the past are made after a trip exploring new edges within the same walk. As explained in previous calculations,
\begin{align}\label{M-eq-3}
M=\left\lfloor\dfrac{m_1}{2}\right\rfloor+\left\lfloor\dfrac{m_2-a}{2}\right\rfloor.
\end{align}
\end{itemize}
If the pair of walks satisfies the above conditions, as we already explained,
\begin{equation}\label{prob-2-in}
\PR(\{w_1\in G\}\cap\{w_2\in G\})=\left(\dfrac{d}{n}\right)^{|w_1\cup w_2|}=\left(\dfrac{d}{n}\right)^{\SL{j\equiv1(\bmod2)}\gamma_j^{(1)}+\SL{j\neq a(\bmod2)}\gamma_j^{(2)}}.
\end{equation}
We now bound the number of respectful pairs of walks $(w_1,w_2)$ in the complete graph $K_n$.\\\\
As in the proof of Lemma \ref{average-vertex}, any trip exploring new edges of length $\gamma$ can be carried out in at most $n^\gamma$ ways, if there is no trip to the past afterwards. If there is a trip to the past afterwards, since the last vertex must be already discovered, there are at most $n^{\gamma-1}\cdot(\ell_1+\ell_2)$ ways to carry out this trip. Moreover, since the pair $(w_1,w_2)$ is respectful, a trip to the past of length $\gamma\le2K\log_d(n)$ can be carried out in at most $d^\gamma\cdot\log(n)^{4K/\delta}$ ways. Observe that any trip in mode 2 or 4 is such a trip, as well as any trip to the past of the walk $w_1$. Also, any trip in mode 5 for $w_2$ can be carried out in at most two ways, since the parts of the two walks intersecting are both Self-Avoiding.\\\\
If $a=0$, the number of ways to choose the pair of starting points of the walks is at most $n^2$. On the other hand, if $a=0$, the starting point of $w_2$ must be visited by $w_1$, so there are at most $\ell_1$ choices for it.
\\\\
We now claim that the number of times we enter mode either 2 or 4 from one of the modes 1,3 and 5 is at most $2\cdot\lfloor\frac{m_1}{2}\rfloor+3\cdot\lfloor\frac{m_2}{2}\rfloor\le3(M+a)$. Indeed, the number of times we enter from modes 1 and 3 are of course at most $\lfloor\frac{m_2}{2}\rfloor$. It is also clear that after being in mode 5, the walk can only shift to mode 4, not 2. So, we need to bound the number of ways this can happen. Because the second part of $w_1$ is Self-Avoiding, for each return to the Non-Backtracking part of $w_1$, there is a point at which it enters and a point from which it leaves. The number of these points is at most $2\cdot\lfloor\frac{m_1}{2}\rfloor$, and for each shift from mode 5 to mode 4 in which the past is from the Non-Backtracking part of $w_1$, only one of the entry/exit points can be points at which the shift occurs. Similarly, for the shifts from mode 5 to mode 4 that occur on an intersection between the Self-Avoiding part of $w_1$ and the Non-Backtracking part of $w_2$, each possible shifting point was an entry or an escape point during the Non-Backtracking part of $w_2$. Therefore, the possible number of shifting points is at most $2\cdot\lfloor\frac{m_2}{2}\rfloor$.\\\\
Finally, it should be noted that due to the second part of $w_1$ being Self-Avoiding, given the $\gamma_j^{(i)}$'s and the lengths of the trips to modes 2 and 4 of $w_2$, the lengths of the trips to mode 5 are also determined. The number of ways to decide the lengths of the trips to modes 2 and 4 is at most $(2K\log_d(n))^{3M+3a}$, as each one of them has length at most $2K\log_d(n)$ and there are at most $3M+3a$ such trips.\\\\
Taking all of these into consideration, we conclude that the expected number of pairs of walks appearing in $G$ satisfying these assumptions is at most
$$n^2\cdot d^{\ell_1+\ell_2-\gamma^{(5)}}\cdot\left(\log(n)^{4K/\delta}\right)^{3M+3a}\cdot\left(\dfrac{d\cdot(\ell_1+\ell_2)}{n}\right)^M\cdot\left(\dfrac{\ell_1}{n}\right)^{a}.$$
The number of solutions to equations (\ref{gamma-eq-4}) and (\ref{gamma-eq-5}) is at most
$$(\ell_1+\ell_2)^M\cdot(K\log(n))^{\lfloor\frac{m_1}{2}\rfloor}\cdot\dfrac{\gamma_p^{\lfloor\frac{m_2}{2}\rfloor}}{\lfloor\frac{m_2}{2}\rfloor!}.$$
Putting everything together,
\begin{align*}
\E\left[\dfrac{1}{n^2}\left(\SL{x_1,x_2}C_{x_1}^{(\ell_1)}\cdot C_{x_2}^{(\ell_2)}\right)\cdot\mathbf{1}_{A_n}\right]\\\le\  \SL{a\in\{0,1\}}\SL{M,m_1,m_2}\SL{\substack{\gamma_p,\alpha\\\alpha\le\gamma_p,3K\log_d(n)}} d^{\ell_1+\ell_2-\gamma_p+\alpha}&\cdot\left(\dfrac{c_0\cdot(\ell_1+\ell_2)^2\cdot\log(n)^{12K/\delta}}{n}\right)^M\\&\cdot\left(\dfrac{\ell_1\cdot\log(n)^{12K/\delta}}{n}\right)^a\cdot(K\log_d(n))^{\lfloor\frac{m_1}{2}\rfloor}\cdot\dfrac{\gamma_p^{\lfloor\frac{m_2}{2}\rfloor}}{\lfloor\frac{m_2}{2}\rfloor!}.\\
\end{align*}
We first examine the contribution from terms for which $m_2\ge3$, or $m_2=2$ and $a=0$. In this case, $M\ge1$ and $\lfloor\frac{m_2}{2}\rfloor\le2M.$ Also,
\begin{align*}
\SL{\substack{\alpha\le\gamma_p\\\alpha\le3K\log_d(n)}}d^{\alpha-\gamma_p}\cdot\dfrac{\gamma_p^{\lfloor\frac{m_2}{2}\rfloor}}{\lfloor\frac{m_2}{2}\rfloor!}&=\SL{\gamma_p}d^{-\gamma_p}\cdot\dfrac{\gamma_p^{\lfloor\frac{m_2}{2}\rfloor}}{\lfloor\frac{m_2}{2}\rfloor!}\SL{\alpha\le\min(\gamma_p,3K\log_d(n))}d^\alpha\\&\le c\SL{\gamma_p\ge6K\log_d(n)}d^{-\gamma_p/2}\cdot\dfrac{\gamma_p^{\lfloor\frac{m_2}{2}\rfloor}}{\lfloor\frac{m_2}{2}\rfloor!}+c\SL{\gamma_p<6K\log_d(n)}\gamma_p^{\lfloor\frac{m_2}{2}\rfloor}\\&\le c\left(\left(\dfrac{4}{\log(d)}\right)^{\lfloor\frac{m_2}{2}\rfloor}+(6K\log_d(n))^{3M}\right)\le(7K\log_d(n))^{4M}.\\
\end{align*}

In the case when $m_2=2$ and $a=1$, we must have $\gamma_p\le\ell_2$, which implies
$$\SL{\substack{\alpha\le\gamma_p\\\alpha\le3K\log(n)}}d^{\alpha-\gamma_p}\cdot\dfrac{\gamma_p^{\lfloor\frac{m_2}{2}\rfloor}}{\lfloor\frac{m_2}{2}\rfloor!}\le\sqrt{n}\cdot\SL{\substack{\alpha\le\gamma_p\\\alpha\le3K\log(n)}}d^{\alpha-\gamma_p}\le c\cdot\ell_2\cdot\log(n).\\$$
In the case $m_2=1$ and $a=0$, we must have $\gamma_p=\alpha=0$, whereas in the case $m_2=1$ and $a=1$, we must have $\gamma_p=\ell_2$ and
$$\SL{\alpha\le\ell_2,3K\log(n)}d^{-\ell_2+\alpha}\le(1-d^{-1})^{-1}.$$
The number of solutions $n_{M,a}$ to (\ref{M-eq-3}) is, just as in the proof of Lemma \ref{average-vertex}, at most
$$
\begin{cases}
	1,\ \text{if}\ M=a=0\\
	2,\ \text{if}\ M=0\ \text{and}\ a=1\\
	4(M+1),\ \text{if}\ M\ge1.
\end{cases}
$$
In all of the cases, for any $a$ and $M$,
$$\SL{m_1,m_2}\SL{\substack{\gamma_p,\alpha\\\alpha\le\gamma_p,3K\log_d(n)}} d^{-\gamma_p+\alpha}\cdot(K\log_d(n))^{\lfloor\frac{m_1}{2}\rfloor}\cdot\dfrac{\gamma_p^{\lfloor\frac{m_2}{2}\rfloor}}{\lfloor\frac{m_2}{2}\rfloor!}\le n_{M,a}\cdot(c_0\log(n)^5)^M\cdot(c_0\cdot\ell_2\cdot\log(n))^a.$$

Due to all of these calculations, and the fact that $(\ell_1+\ell_2)^2/n\le e^{-\sqrt{\log(n)}},$
\begin{align*}
&\E\left[\dfrac{1}{n^2}\left(\SL{x_1,x_2}C_{x_1}^{(\ell_1)}\cdot C_{x_2}^{(\ell_2)}\right)\cdot\mathbf{1}_{A_n}\right]\\\le\ &
\ d^{\ell_1+\ell_2}\SL{a\in\{0,1\}}\SL{M\ge0}\left(\dfrac{c_0\log(n)^{12K/\delta+5}}{e^{\sqrt{\log(n)}}}\right)^M\cdot\left(\dfrac{c_0\log(n)^{12K/\delta+1}}{e^{\sqrt{\log(n)}}}\right)^a\cdot n_{M,a}\\=\ &(1+o(1))\cdot d^{\ell_1+\ell_2},
\end{align*}
which exactly implies (\ref{sec-mom-avg}), as we explained.\\\\
Combining this relation with (\ref{1-mom-avg}) and Chebyshev's inequality, we see that if we set
$$Z_n:=\dfrac{1}{n}\SL{x}\SL{\ell=1}^{\sqrt{n}/e^{\sqrt{\log(n)}}}d^{-\ell}\cdot C_x^{(\ell)}\cdot\dfrac{1}{\ell},$$
we get
$$\PR\left(Z_n>2\cdot\E(Z_n)\right)\le\PR(A_n^c)+\dfrac{\E\left[\left(Z_n-\E(Z_n)\right)^2\cdot\mathbf{1}_{A_n}\right]}{\E(Z_n\cdot\mathbf{1}_{A_n})^2}=o(1).$$
The proof is complete.
\end{proof}

\end{document}